\documentclass[11pt]{article}
\usepackage[utf8]{inputenc}

\usepackage{mathtools}
\usepackage{amsthm}
\usepackage{amsmath,amsfonts,amssymb,latexsym,amsthm,dsfont,amsbsy,stmaryrd}
\usepackage{caption}
\usepackage[left=3cm,right=3cm,top=1.8cm,bottom=2.3cm]{geometry}

\usepackage{tikz}
\usetikzlibrary{shapes,positioning}
\usepackage{xcolor,graphicx}
\usepackage{enumitem}
\usepackage[colorlinks=True]{hyperref}

\newtheorem{thm}{Theorem}[section]
\newtheorem{cor}[thm]{Corollary}
\newtheorem{prop}[thm]{Proposition}
\newtheorem{lem}[thm]{Lemma}

\newtheorem{rk}[thm]{Remark}

\let \ssection=\section
\renewcommand{\section}{\setcounter{equation}{0}\ssection}

\newcommand{\ds}[1]{{\displaystyle{#1 }}}
\renewcommand{\leq}{\leqslant}
\renewcommand{\geq}{\geqslant}
\newcommand{\ind}{\mathds{1}}

\newcommand{\dd}{{\mathrm d}}
\newcommand{\ee}{{\mathrm e}}
\newcommand{\Span}{\mathrm{Span}}
\newcommand{\Supp}{\mathrm{Supp}}

\newcommand{\dC}{\mathbb{C}}
\newcommand{\dE}{\mathbb{E}}

\newcommand{\dP}{\mathbb{P}}
\newcommand{\dR}{\mathbb{R}}
\newcommand{\dV}{\mathbb{V}}
\newcommand{\dZ}{\mathbb{Z}}

\newcommand{\cB}{\mathcal{B}}
\newcommand{\cC}{\mathcal{C}}

\newcommand{\cL}{\mathcal{L}}
\newcommand{\cM}{\mathcal{M}}
\newcommand{\cN}{\mathcal{N}}
\newcommand{\cP}{\mathcal{P}}

\newcommand{\cR}{\mathcal{R}}
\newcommand{\cU}{\mathcal{U}}
\newcommand{\cW}{\mathcal{W}}

\newcommand{\bD}{\mathbf{D}}
\newcommand{\bL}{\mathbf{L}}
\newcommand{\bM}{\mathbf{M}}
\newcommand{\bS}{\mathbf{S}}
\newcommand{\bU}{\mathbf{U}}
\newcommand{\bW}{\mathbf{W}}
\newcommand{\bX}{\mathbf{X}}
\newcommand{\bY}{\mathbf{Y}}
\newcommand{\bZ}{\mathbf{Z}}
\newcommand{\be}{\mathbf{e}}
\newcommand{\bbf}{\mathbf{f}}
\newcommand{\bm}{\mathbf{m}}
\newcommand{\bu}{\mathbf{u}}
\newcommand{\bv}{\mathbf{v}}
\newcommand{\bw}{\mathbf{w}}

\newcommand{\by}{\mathbf{y}}
\newcommand{\bz}{\mathbf{z}}

\newcommand{\ABS}[1]{{{\left| #1 \right|}}} 
\newcommand{\BRA}[1]{{{\left\{#1\right\}}}} 
\newcommand{\SCA}[1]{{{\left<#1\right>}}} 
\newcommand{\NRM}[1]{{{\left\| #1\right\|}}} 
\newcommand{\PAR}[1]{{{\left(#1\right)}}} 
\newcommand{\SBRA}[1]{{{\left[#1\right]}}} 

\title{A fixed-point equation approach for the superdiffusive elephant random walk}
\author{Hélène Guérin\thanks{Université du Québec à Montréal (UQAM), Département de Mathématiques, Canada. Funded by Natural Sciences and Engineering Research Council of Canada (NSERC) discovery grant (RGPIN-2020-07239); 
\texttt{guerin.helene@uqam.ca}},\ Lucile Laulin\thanks{Nantes Université, Laboratoire de Mathématiques Jean Leray, France. Funded by Centre Henri Lebesgue,
programme ANR-11-LABX-0020-01; \texttt{lucile.laulin@math.cnrs.fr}}\; and Kilian Raschel\thanks{Université d'Angers, CNRS, Laboratoire Angevin de Recherche en Mathématiques, France. This project has received funding from the European Research Council (ERC) under the European Union's Horizon 2020 research and innovation programme under the Grant Agreement No. 759702 and from Centre Henri Lebesgue,
programme ANR-11-LABX-0020-01; \texttt{raschel@math.cnrs.fr}}}
\date{\today}

\begin{document}

\maketitle
\begin{abstract}
We study the elephant random walk in arbitrary dimension $d\geq 1$. Our main focus is the limiting random variable appearing in the superdiffusive regime.
Building on a link between the elephant random walk and P\'olya-type urn models, we prove a fixed-point equation (or system in dimension two and larger)\ for the limiting variable. Based on this, we deduce several properties of the limit distribution, such as the existence of a density with support on $\mathbb R^d$ for $d\in\{1,2,3\}$, and we bring evidence for a similar result for $d\geq 4$. We also investigate the moment-generating function of the limit and give, in dimension $1$, a non-linear recurrence relation for the moments.
\end{abstract}

\begin{quote}

\end{quote}

\section{Introduction and main results}
\paragraph{Context}
The one-dimensional elephant random walk (ERW)\ was introduced by Schütz and Trimper \cite{Schutz2004} in 2004, in order to see how memory could induce subdiffusion in random walk processes. It turned out that the ERW is, in fact, always at least diffusive. Nevertheless, the easy definition of the process together with the underlying deepness of the results has lead to a great interest from mathematicians over the last two decades. 

The process $(S_n)_{n\geq 0}$ is defined as follows. We denote by $(X_n)_{n\geq 0}$ its successive steps.
The elephant starts at the origin at time zero: $S_0 = 0$. For the first step $X_1$, the elephant moves one step to the right with probability 
$q$ or one step to the left with probability $1-q$, for some $q$ in $[0,1]$. The next steps are performed by choosing uniformly at random an integer $k$ among the previous times. Then the elephant moves exactly in the same direction as at time $k$ with probability $p\in[0,1]$, or in the opposite direction with probability $1-p$.
In other words, defining for all $n \geq 1$, 
\begin{equation}
\label{STEPS}
   X_{n+1} = \left \{ \begin{array}{lll}
    +X_{k} &\text{with probability  $p$,} \\[1em]
    -X_{k} &\text{with probability $1-p$,}
   \end{array} \right.
\end{equation}
with $k\sim\cU\BRA{1,\ldots,n}$, the position of the ERW at time $n+1$ is given by
\begin{equation}
\label{POSERW}
   S_{n+1}=S_{n}+X_{n+1}.
\end{equation}
The probability $q$ is called the first step parameter and $p$ the memory parameter of the ERW. An ERW trajectory is sampled in Figure~\ref{fig:simulations_ERW}.

A wide range of literature is now available on the ERW and its extensions in dimension $d=1$, see for instance
\cite{BaurBertoin2016,Coletti2017,Cressoni2013,Cressoni2007,daSilva2013,Kursten2016}. In dimension $d \geq 1$, the multi-dimensional ERW (MERW)\ was later introduced by Bercu and Laulin \cite{BercuLaulin2019}. 
Back to dimension $1$, the behavior of the process, in particular its dependency on the value of $p$ with respect to the critical value $3/4$, is now well understood. In the diffusive regime $p< 3/4$ and the critical regime $p=3/4$, a strong law of large numbers and a central limit theorem for the position, properly normalized, were established, see \cite{BaurBertoin2016,Coletti2017,ColettiN2017,Schutz2004} and the more recent contributions \cite{BercuHG2019,Coletti2019,Fan2020,Gonzales2020,roy2024elephant,roy2024phase,Vazquez2019}. The main changes between the two regimes are the rate of the associated convergences.

\begin{figure}[ht!]
    \centering
    \includegraphics[width=7cm,height=5.5cm]{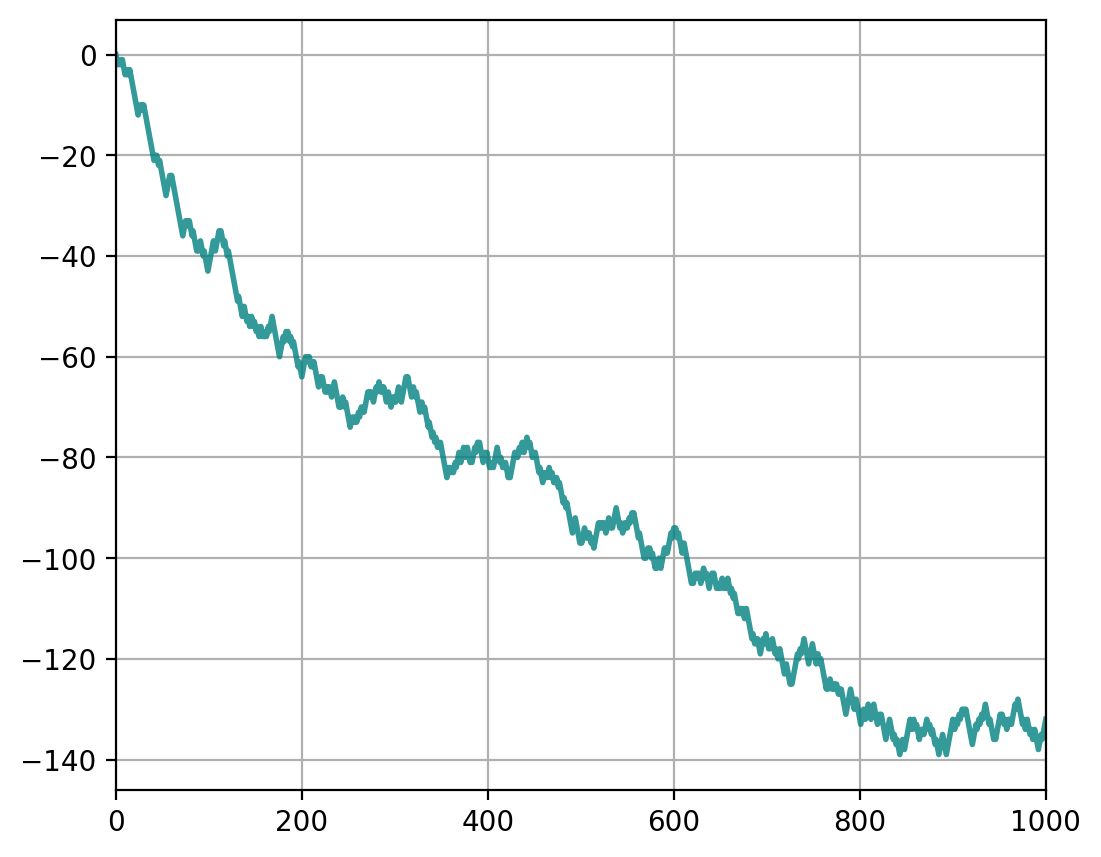}
    \caption{A trajectory of the Elephant random walk in the superdiffusive regime until time $n=1000$, with $p=0.87$ and $q=0.9$.} 
    \label{fig:simulations_ERW}
\end{figure}

The superdiffusive regime $p>3/4$ is even more intriguing. Introduce 
\begin{equation*}
    a:=2p-1.
\end{equation*}
It has been established that
\begin{equation}
\label{ASCVGS}
    \lim_{n \rightarrow \infty} \frac{S_n}{n^{a}}=L_q \quad \text{a.s.},
\end{equation}
where $L_q$ is a non-degenerate, non-Gaussian random variable, see \cite{BaurBertoin2016,Bercu2018,Coletti2017}. This result has been extended to MERW in \cite{BercuLaulin2019}, with a rate of convergence depending on the dimension $d$. In dimension $1$, it has been proved in \cite{Kubota2019} that the fluctuations of the ERW around its limit $L_q$
are Gaussian, as indeed
\begin{equation}
\label{ANS}
\sqrt{n^{2a-1}} \Bigl(\frac{S_n}{n^{a}} -L_q \Bigr) \underset{n\to \infty}{\overset{\cL}\longrightarrow} \cN \Bigl(0, \frac{1}{2a-1}\Bigr).
\end{equation}
See \cite{Bertenghi2020} for a similar result in dimension $d\geq 2$. These results were established thanks to a martingale approach.

Little is known about the limit variable $L_q$, although it is a natural question to try to find out more about its law. Indeed, among the open questions concerning the (M)ERW over the last two decades, two stood out: the law of the random variable $L_q$ arising in the superdiffusive regime, and the recurrence and transience of the MERW, depending on the regime. In this work we give a sharp answer to the first question, as detailed below. The second question was first addressed by Bertoin in dimension $d=1$ in \cite{Bertoinb2021}, establishing that the ERW is recurrent in the diffusive regime and transient in the superdiffusive regime. However, until recently, this remained unresolved for dimensions $d\geq 2$. Qin in \cite{qin2023recurrence} recently provided an answer, showing that in dimension $d=2$, the MERW behaves similarly to the ERW, whereas in dimensions $d\geq3$, it is always transient, see also \cite{curien2023recurrence} for an alternative proof of the recurrence in the diffusive regime for $d=2$.

\paragraph{Our contributions}

In this article, we are interested in studying the distribution of the asymptotics of the superdiffusive ERW, first in  dimension $1$ (the variable $L_q$ appearing in \eqref{ASCVGS}), then in higher dimension. Indeed, before writing this paper, a lot of natural questions regarding this limiting random variable (such as the existence of a density, an explicit formula for the moments, the moment problem, the finiteness of the moment-generating function, etc.)\ were still open. While it is generally very difficult to obtain results for the ERW in dimension $d\geq 2$, an interesting feature of the present work is that the techniques work in all dimensions, even though we don't have a complete proof in dimension $d\geq 4$: our result is proved in dimensions $2$ and $3$, and we give a method for $d\geq4$ based on a (linear algebra)\ analysis of a certain Krylov space. The main idea of this paper is to establish fixed-point equations  related to  the limit distribution of the ERW.

\paragraph{Dimension $1$}

Using a connection between the ERW and P\'olya-type urns, we first establish in Section~\ref{S-FPE} that in dimension $d=1$, if we impose $q=1$, which means that the elephant first goes to the right, then the random variable $L_1$ satisfies the following fixed-point equation, from which everything will follow:
\begin{thm}
\label{thm:fixed-point-dim1}
Assume $a=2p-1>\frac{1}{2}$. Let $L_1$ be the limit variable appearing in \eqref{ASCVGS} with $q=1$. 
The distribution of $L_1$ satisfies the fixed-point equation 
  \begin{equation}
 \label{eq:fixed-point-X-only}
      L_1 \overset{\cL}= V^{a} L_1^{(1)} + (2\xi_p-1)(1-V)^{a}L_1^{(2)},
 \end{equation}
  where
 \begin{itemize}
     \item $V$ is a uniformly distributed random variable on $[0,1]$;
     \item $\xi_p$ is a Bernoulli distributed random variable with parameter $p$;
     \item $L_1^{(1)}$ and $L_1^{(2)}$ have the same distribution as $L_1$; 
     \item all the variables are independent.
 \end{itemize}
\end{thm}
The proof of  Theorem~\ref{thm:fixed-point-dim1} is detailed in Section~\ref{sec:tree-struct-dim1}. As we shall see in Remark~\ref{rk:loiLX}, one has 
\begin{equation}
    \label{eq:link_L_1_L_q}
    L_q \overset{\cL}= (2\xi_q -1)L_1,
\end{equation}
where as above $\xi_q$ denotes a Bernoulli random variable with parameter $q$, independent of $L_1$.

The connection between ERW and P\'olya-type urns was first used by Baur and Bertoin \cite{BaurBertoin2016} in order to obtain functional convergence for the ERW, thanks to the work of Janson \cite{Janson2004} on generalized P\'olya urns. The originality of our work lies in the use of tree decomposition: this connection between the ERW and urn models is different from what has been done until now. This point of view together with the initial urn connection has enabled us to establish the fixed-point equation~\eqref{eq:fixed-point-X-only}.

Equations such as \eqref{eq:fixed-point-X-only} have already been used to study the limit of generalized P\'olya urn processes with deterministic replacement, see 
\cite{Li1992PTRF, Liu2001,Chauvin-Liu-Pouyanne2014, Chauvin2015}. 
In the same spirit, in our random replacement context, we will deduce various properties of the limit distribution from the fixed-point equation that they satisfy. We mention that there is a wide literature on the study of solutions of general fixed-point equations, see e.g.\ \cite{Durrett-Liggett1983,ABM2012,Iksanov-Meiners2015,MM2017} and the references therein.

The next result, proved in Section~\ref{sec:tree-struct-dim1}, ensures the existence and uniqueness of the solution to \eqref{eq:fixed-point-X-only}, with constraints on the first and second moments (we already know by \cite{Janson2004}, recalled here in Theorem \ref{thm:Janson-dim1}, that $\dE[L_1]=1/\Gamma(1+a)$).
\begin{thm}
\label{thm:dim1_solution_dist_eq}
Let $a\in(\frac{1}{2},1) $.   The random variable $L_1$ is the unique solution of the distributional equation \eqref{eq:fixed-point-X-only} having $\frac{1}{\Gamma(1+a)}$ as expectation and finite second moment.
\end{thm}

Further, in Section~\ref{sec:existence-density-dim1}, we solve a conjecture on the existence of a density of the limit random variable:
 \begin{thm}
 \label{thm:density_dim_1}
    For any $q\in[0,1], \, a\in(\frac{1}{2},1) $, the random variable $L_q$ is absolutely continuous with respect to Lebesgue measure. Its density is a positive, bounded, smooth function on $\dR$.
 \end{thm}
 
Few illustrations of that density function are presented in Figure~\ref{fig:simulations}, Section~\ref{sec:approx-density}.

 We also prove in Section~\ref{sec:moments-dim1} that the random variable $L_1$ has finite moments of all order and that its probability distribution is determined by its moments (see Theorem~\ref{thm:exp-moments}). Another interesting consequence of the fixed-point equation \eqref{eq:fixed-point-X-only} is to give a simple way to compute all moments of $L_1$ by induction.  So far, only the first four moments were known, see \cite{Bercu2018}.

\begin{thm}\label{thm:moments}
    Let $a\in(\frac{1}{2},1) $. The moments of $L_1$ are given by the following recursive equation.
    Let $(m_k)_{k\geq 1}$ be defined by $m_1=1$ and, for $k\geq 2$,
    \begin{equation}\label{eq:recursion_moments} m_{k}=\frac{1}{ka-c_k}\sum_{j=1}^{k-1}c_jm_jm_{k-j},\end{equation}
    where $c_k=1$ for even $k$ and $c_k=a$ for odd $k$.
 Then, for any $k\geq 1$,
    \begin{equation}\label{eq:L_moments}
    \dE[L_1^k]=\frac{(k-1)!}{a\Gamma(ka)}m_k,
    \end{equation}
    and the moment-generating function of $L_1$ is given by, for $t\in\dR$,
    \[
\dE[\ee^{tL1}]=\sum_{k\geq 0}\frac{m_k}{\Gamma(ka+1)}t^{k} .
    \]
\end{thm}
This theorem is proved at the end of Section~\ref{sec:moments-dim1}. Using \eqref{eq:link_L_1_L_q}, we immediately deduce the moments of $L_q$: $\dE[L_q^k] = (q+(-1)^k (1-q))\dE[L_1^k]$.
Finally, we prove that the moment-generating function of $L_1$ and $L_1^2$ are both finite everywhere on $\mathbb R$, see Corollaries~\ref{cor:laplaceL} and \ref{cor:laplaceL^2}. 

At the end of the study in dimension $1$, we give in Section~\ref{sec:randomtrees} a series representation \eqref{eq:L-sum_details} of the random variable $L_q$ using classical distributions, which a direct consequence of the interpretation of ERW in terms of random trees (see \cite{Baur2015,Businger2018} for more details on the subject).

\paragraph{Dimension $d\geq 2$}

Section~\ref{S-HD} is devoted to the multidimensional case $d\geq2$, for which we derive similar results as in dimension $1$, up to some peculiarities and subtleties. A great feature of our way of using the link between the elephant random walk and the P\'olya-type urns model in dimension $1$ is that it generalizes nicely to higher dimensions, whereas most work on MERW has had to use stronger tools to break out of dimension 1. We first recall the definition of the multidimensional ERW in Section~\ref{subsec:MERW}, in particular the result \cite{BercuLaulin2019} giving the existence of a limit variable $\bL$ in the superdiffusive regime. 

In Section~\ref{subsec:urn_process_MERW}, similarly to the dimension $1$ case, we associate to the MERW a P\'olya-type urn process with $2d$ different colors (representing the $2d$ possible directions of the MERW).  In the superdiffusive limit, the asymptotic P\'olya urn process obtained by \cite{Janson2004} (recalled here in Theorem~\ref{thm:Janson-dimd}) can be characterized by two different fixed-point systems, given by \eqref{eq:fixed-point-X-color-vector} and \eqref{eq:fixed-point-W-dimd}, analogous to Theorem~\ref{thm:fixed-point-dim1} in dimension $1$, which will both be used in our study. These fixed-point equations were more delicate to obtain than in dimension $1$. It is also important to remark that as soon as the ERW is concerned, we only need a partial information from the urn reformulation, as the MERW has natural dimension $d$, while the urn model has dimension $2d$. The link between the asymptotic $\bW$ of the urn process and the limit MERW $\bL$ is formalized in Section~\ref{subsec:link_urn_MERW}.
 
Existence and uniqueness of the solution of the first fixed-point system~\eqref{eq:fixed-point-X-color-vector} is studied in Section~\ref{sec:fixed-point-dimd}.
In Section~\ref{subsec:density_d}, working with the second fixed-point system~\eqref{eq:fixed-point-W-dimd}, we  give various properties of the limit variable $\bW$ in the urn interpretation, when its support is the full space. In particular we prove existence of a smooth density in Theorem~\ref{thm:density-W-dimd}. We conjecture that our hypothesis on the support is systematically satisfied, and in Section~\ref{sec:support} we show that indeed it holds in dimension $2$ and $3$. We further give a method to check this hypothesis in dimension $d\geq 4$, based on the analysis of certain Krylov spaces.

In Section~\ref{subsec:application_density_MERW} we show how to transfer from the urn limit variable $\bW$ new properties on the limit variable $\bL$ of the MERW, and we deduce the following theorem in dimension $d=2$ and $d=3$.
\begin{thm}\label{cor:densityL-dim3}
We assume that $d=2$ or $d=3$. Let $\frac{1}{2}<a<1$ with $a=\frac{2dp-1}{2d-1}$.
    The limit $\bL$ of a superdiffusive  elephant random walk in dimension $d$ 
has a smooth bounded density on $\dR^d$, and its support is $\dR^d$.
\end{thm}

Finally, in Section~\ref{sec:moments-dimd} we prove that the limit MERW variable has finite moments of all order, that its probability distribution is determined by its moments, and that in addition the moment-generating function of $\bL$ and $\NRM{\bL}^2$ are globally well defined.

\section{One-dimensional superdiffusive limit using P\'olya-type urns}
\label{S-FPE}

\subsection{Connection between the ERW and P\'olya-type urns}
\label{subsec:connection_ERW_Polya}
We now present the urn approach of the ERW in dimension $1$, but everything can be extended to dimension $d$ to obtain a more general fixed-point equation. The higher dimensional case is studied in full details in Section~\ref{S-HD}.

Let $\PAR{\bU(n)}_{n\geq 0}$ be a discrete-time urn with balls of two colors, say red and blue. The composition of
the urn at time $n\geq 0$ is given by a vector $\bU(n) =(R_n, B_n)^T$, where $R_n$ (resp.\ $B_n$) stands for the number of red (resp.\ blue) balls at time $n$. The initial composition of the urn is $(1,0)$ with probability $q$ or $(0,1)$ with probability $1-q$. The composition of urn then evolves as follows. At any time $n \geq 1$, a ball is drawn uniformly at random, its color is observed, then it is returned to the urn together with a ball of the same color with probability $p$, or with a ball of the other color with probability $1-p$. 

Hence, the replacement matrix $A$ is defined by
\begin{equation}
\label{def:A}
    A=\left\{\begin{array}{ll}
    \begin{pmatrix} 1 & 0 \\ 0 & 1 \end{pmatrix} & \textnormal{with probability } p,\medskip\\
     \begin{pmatrix} 0 & 1 \\ 1 & 0 \end{pmatrix} & \textnormal{with probability } 1-p.
    \end{array}\right.
\end{equation}
The composition of the urn at time $n+1$ satisfies
\begin{equation*}
     \bU({n+1}) = \bU(n) + A_{n+1} \begin{pmatrix}\varepsilon_{n+1} \\ 1- \varepsilon_{n+1}\end{pmatrix},
\end{equation*} 
where $\varepsilon_{n+1}$ is equal to $1$ if a red ball is picked and to 0 otherwise, and $(A_n)_{n\geq 1}$ is a sequence of i.i.d.\ random matrices with same distribution as $A$.

The behavior of the process is deeply linked to the spectral decomposition of the mean replacement matrix $\dE[A]$, given by 
\begin{equation*}
   \dE[A] = \begin{pmatrix} p & 1-p \\ 1-p & p \end{pmatrix}.
\end{equation*}
The eigenvalues of $\dE[A]$ are $\lambda_1 = 1$ and $\lambda_2 = 2p-1$,
and the corresponding unit eigenvectors are
\begin{equation*}
\bv_1 = \frac{1}{2}\begin{pmatrix}
    1\\1
\end{pmatrix}\quad\text{and} \quad \bv_2 = \frac{1}{2}\begin{pmatrix}
    1\\-1
\end{pmatrix}.
\end{equation*} 
It is well known, see \cite{Chauvin2011,Flajolet2006,Flajolet2005,Janson2004}, 
that the asymptotics of the urn depends on the ratio $a=\lambda_2/\lambda_1 =2p-1 $ with respect to $1/2$. In particular, applied to our special case we have the following large urn convergence result \cite{Janson2004,BaurBertoin2016}.  
\begin{thm}[\cite{Janson2004,BaurBertoin2016}]
\label{thm:Janson-dim1}
Let $U(0) = \alpha=(\alpha_1,\alpha_2)\in\BRA{0,1}^2$ be such that $\alpha_1+\alpha_2=1$. If $a=2p-1 >1/2$, then it holds
\begin{equation*}
    \lim_{n\to\infty} \frac{\bU(n) - n\bv_1}{n^a } = W_{\alpha}\bv_2 \quad\text{a.s.},
\end{equation*}
where $W_{\alpha}$ is a non-degenerate random variable such that
\begin{equation*}
    \dE[W_{\alpha}]=\ds{\frac{\alpha_1-\alpha_2}{\Gamma(1+a)}}\quad \text{and} \quad \dE[W_{\alpha}^2]=\ds{\frac{1}{(2a-1)\Gamma(2a)}}.
\end{equation*}
In particular, 
\begin{equation*}
    \lim_{n\to\infty} \frac{R_n - B_n}{n^a} = W_{\alpha} \quad\text{a.s.}
\end{equation*}
\end{thm}

The connection to the ERW model is straightforward.  Let $(S_n)_{n\geq 0}$ denote the ERW started from $S_0 = 0$ and such that $S_1 = R_1 - B_1$; then for every $n\geq 1$,
\begin{equation*}
   S_n \overset{\cL}= R_n - B_n.
\end{equation*}
 In other words, the difference between the number of red and blue balls in the urn behaves like an ERW with first step equal to $R_1 - B_1$ (the red balls corresponding to steps to the right, and the blue ball corresponding to left jumps).
 Hence, if $W_{\alpha}$ denotes the limiting random variable for the urn process started from the composition vector $\alpha$, we have the equality in distribution
 \begin{equation}
 \label{eq:L-urn}
   L_q \overset{\cL}= \xi_q W_{(1,0)} + (1-\xi_q) W_{(0,1)},
\end{equation}
where $\xi_q$ is a Bernoulli variable with parameter $q$, and $\xi_q$, $W_{(1,0)}$ and $W_{(0,1)}$ are mutually independent. In particular, when the elephant starts with a right jump, we have $L_1\overset{\cL}{=}W_{(1,0)}$.
This explains how the asymptotic behavior of the ERW is determined by the spectral decomposition of the (mean)\ replacement matrix of the corresponding urn. 

Finally, we briefly recall a nice result on $L_q$, obtained by using a connection between the ERW and random recursive trees on which a Bernoulli bond percolation has been performed (see \cite{Kursten2016,Businger2018} for more details on the subject). It turns out that  results on cluster sizes after percolation, e.g.\ from \cite[Lem.~3.3]{Baur2015}, can lead to the following decomposition of the random variable $L_q$:
\begin{equation}
\label{eq:L-sum}
   L_q=
    C_1Z_1^{(q)} + \sum_{j=2}^{+\infty} C_j Z_j \quad\text{a.s.},
\end{equation}
where $Z_1^{(q)}$ has Rademacher distribution with parameter $q$ (denoted by $\cR(q)$), $(Z_j)_{j\geq2}$ are i.i.d.\ with $\cR(1/2)$ distribution, $C_1$ has a Mittag-Leffler distribution with parameter $a$,
and $(C_j)_{j\geq 2}$ is a sequence of dependent, almost surely positive random variables (their explicit expression is given in Proposition~\ref{prop:rec-trees-dim1}, Equation~\eqref{eq:L-sum_details}). This decomposition ensures that the random variable $L_q$ has a continuous distribution (no atoms). Although we find this expression very elegant, we will not use it because we have not been able to prove the existence of a density on $\dR$ for the distribution of $L_q$ from it. We will instead focus on the urn-connection and the underlying fixed-point equations to obtain properties of the superdiffusive limit. 

\subsection{Tree-structure and fixed-point equations}\label{sec:tree-struct-dim1}

We will prove that both random variables $W_{(1,0)}$ and $W_{(0,1)}$, introduced in Theorem~\ref{thm:Janson-dim1}, admit a density, by showing that their characteristic functions are integrable. Our method is strongly inspired by the work of Chauvin et al.~\cite{Chauvin2015} for urn processes with deterministic replacement.

The tree structure of the urn process is as follows: a sequence of trees $(T_n)_{n\geq 0}$ grows at each drawing from the urn. At time $0$, there is a red ball in the urn with probability $q$, or a blue ball with probability $1-q$. Then the tree starts from a red node with probability $q$ or a blue node with probability $1-q$:
\begin{center}
\begin{tikzpicture}
\tikzstyle{bleu}=[circle,fill=blue!70]
\tikzstyle{rouge}=[circle,fill=red!70]
\node [rouge, label =left: \footnotesize with probability $q$] at (0,0) {};
\node [bleu, label =right: \footnotesize with probability $1-q$] at (2,0) {};
\end{tikzpicture}
\end{center}
At time $n$, each leaf in the tree represents a ball in the urn. When a leaf is chosen (i.e., a ball is drawn), it becomes an internal node and gives birth to $2$ balls of the same color with probability $p$, or $1$ ball of each color with probability $1-p$. Due to the urn process correspondence, leaves are uniformly chosen among the leaves of the tree

At time $1$, the tree is decomposed into two subtrees (or a forest with two trees); 
the three possible decompositions of the tree into two subtrees at time $1$ are illustrated on Figure~\ref{fig:subtrees-t1}.
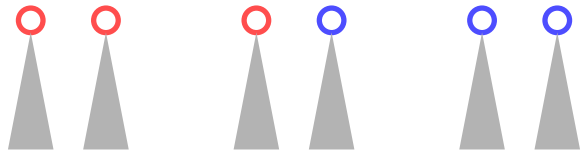
\begin{figure}[ht!]
\centering
\begin{tikzpicture}
\tikzstyle{sousarbre}=[isosceles triangle,fill=black!30,isosceles triangle apex angle=22,rotate=90,scale=1.7]
\tikzstyle{bleu}=[circle,fill=blue!70]
\tikzstyle{rouge}=[circle,fill=red!70]
\tikzstyle{rougechoisie}=[circle,draw=red!70,line width=2pt,scale=0.9]
\tikzstyle{bleuchoisie}=[circle,draw=blue!70,line width=2pt,scale=0.9]
\node[rougechoisie] at (0,0) {}
    child { node[sousarbre, draw=none] {} edge from parent[draw=none]};
\node [rougechoisie] at (1,0) {}
    child { node[sousarbre, draw=none] {} edge from parent[draw=none]};
\node [rougechoisie] at (3,0) {}
    child { node[sousarbre, draw=none] {} edge from parent[draw=none]};
\node [bleuchoisie] at (4,0) {}
    child { node[sousarbre, draw=none] {} edge from parent[draw=none]};
\node [bleuchoisie] at (6,0) {}
    child { node[sousarbre, draw=none] {} edge from parent[draw=none]};
\node [bleuchoisie] at (7,0) {}
    child { node[sousarbre, draw=none] {} edge from parent[draw=none]};
\end{tikzpicture}
\caption{The three possible subtrees at time $n=1$.}\label{fig:subtrees-t1}
\end{figure}
Two possibilities for the tree decomposition at time $4$ are displayed on Figure~\ref{fig:tree-time4} below, depending on the first ball added to the urn (depending on the color of the root node). The balls that have been picked are represented by the empty discs, while the balls that can be picked at next times are plain discs (as already said, they correspond to the leaves of the tree). On Figure~\ref{fig:tree-time4}, after time $4$, there are $5$ balls in the urn. 
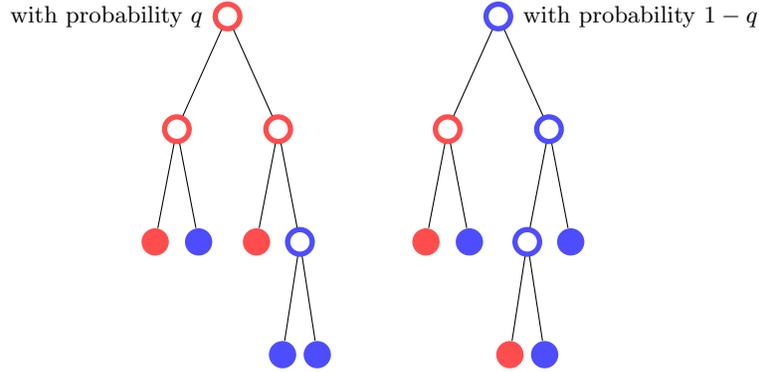
\begin{figure}[ht!]
\centering
\begin{tikzpicture}
[level 1/.style={sibling distance=3.5em},
level 2/.style={sibling distance=1.5em},
level 3/.style={sibling distance=1.2em}]
\tikzstyle{bleu}=[circle,fill=blue!70]
\tikzstyle{rouge}=[circle,fill=red!70]
\tikzstyle{rougechoisie}=[circle,draw=red!70,line width=2pt,scale=0.9]
\tikzstyle{bleuchoisie}=[circle,draw=blue!70,line width=2pt,scale=0.9]
\node [rougechoisie, label =left: \footnotesize with probability $q$] at (0,0) {}
    child { node[rougechoisie] {}
        child { node[rouge] {}}
        child { node[bleu] {}}
        }
    child { node[rougechoisie] {}
        child { node[rouge] {}}
        child { node[bleuchoisie] {}
            child { node[bleu] {}}
            child { node[bleu] {}}
        }
    }
    ;
\node [bleuchoisie, label =right: \footnotesize with probability $1-q$] at (3.6,0) {}
    child { node[rougechoisie] {}
        child { node[rouge] {}}
        child { node[bleu] {}}
        }
    child { node[bleuchoisie] {}
        child { node[bleuchoisie] {}
            child { node[rouge] {}}
            child { node[bleu] {}}
            }
        child { node[bleu] {}}
    }
    ;
\end{tikzpicture}
\caption{Two examples of tree at time $4$.}\label{fig:tree-time4}
\end{figure}

For any $n\geq1$, denote by $D_1(n)$ the number of leaves  at time $n$ of the first subtree (starting on the left node at time $1$), and similarly $D_2(n)$ the number of leaves  at time $n$ of the second subtree (starting on the right node at time $1$). At time $n$, we note that $D_1(n)+D_2(n)=n+1$, and the number of drawings in the $k$-th subtree is $D_k(n)-1$. These numbers represent the time inside each subtree. 

We now describe the evolution of $D(n)$. Remember that the balls of the whole urn are uniformly drawn at any time and notice that at each drawing in the $k$-th subtree, $D_k(n)$ increases by $1$.
The random vector $\bD(n)=(D_1(n), D_2(n))$ can indeed be seen as a classical $2$-color  urn process (each color modeling each subtree),
with $I_2$ as (deterministic) replacement matrix and $(1,1)$ as initial composition vector.

We recall that  $\bU_{(1,0)}$ is the urn process starting from a red ball. Then the tree will start from a red node. The number of red/blue leaves in the whole tree is equal to the sum of the number of red/blue leaves in the first subtree and the number of red/blue leaves in the second subtree.
We also easily observe that the first subtree of $\bU_{(1,0)}$ always starts from a red node, and the second subtree starts from a red node with probability $p$ and a blue node with probability $1-p$.

Gathering these arguments, the distribution of the urn process $\bU_{(1,0)}$ can be described the following way: consider simultaneously
\begin{itemize}
     \item an original $2$-color urn process $\bD = (D_1, D_2)$ having the identity matrix as (deterministic) replacement matrix and $(1,1)$ as initial condition;
     \item for any $k\in\{1,2\}$, an urn process $\bU^{(k)}_{(1,0)}$ having $A$ in \eqref{def:A} as (random) replacement matrix and $(1,0)$ as initial condition;
     \item an urn process $\bU^{(2)}_{(0,1)}$ having $A$ in \eqref{def:A} as (random) replacement matrix and $(0,1)$ as initial condition;
     \item a Bernoulli random variable $\xi_p$ with parameter $p$ (modeling the first replacement);
\end{itemize}  
 all these processes being independent of each other. Then, the process $\bU_{(1,0)}=(\bU_{(1,0)}(n))_{n\geq 1}$ satisfies the distributional equality, for all $n\geq 1$:
 \begin{equation*}
     \bU_{(1,0)}(n)\overset{\cL}= \bU^{(1)}_{(1,0)}\PAR{D_1(n)-1} + \xi_p \bU^{(2)}_{(1,0)}\PAR{D_2(n)-1) +(1-\xi_p} \bU^{(2)}_{(0,1)}\PAR{D_2(n)-1}.
 \end{equation*}
 A similar equation can be obtained for $\bU_{(0,1)}$.

 Since $\bD = (D_1, D_2)$ has the same behavior as a classical urn process, we immediately have
 \begin{equation}\label{eq:conv-D}
     \lim_{n\to\infty} n^{-1}\bD(n) = \bZ \quad\text{a.s.},
\end{equation}
where $\bZ=(V,1-V)$ has Dirichlet$(1,1)$ distribution, with $V\sim\cU([0,1])$ a uniform random variable on $[0,1]$ (see \cite{polya1930} or \cite[Thm~2.1]{freedman1965}).

Using the above considerations together with Theorem~\ref{thm:Janson-dim1}, we deduce the following result.

\begin{thm}
\label{thm:system_distrib}
Assume $a\in\PAR{\frac{1}{2},1}$. Let $W_{(1,0)}$ and $W_{(0,1)}$ be the elementary limit distributions of a large two-color Pólya urn process with random replacement matrix $A$ given by \eqref{def:A} and ratio $a >1/2$, introduced in Theorem~\ref{thm:Janson-dim1}. Then $W_{(1,0)}$ and $W_{(0,1)}$ satisfy the distributional equation system
 \begin{equation}
 \label{eq:fixed-point-XY}
     \begin{cases}
     W_{(1,0)} \overset{\cL}= V^a  W_{(1,0)}^{(1)} + \xi_p (1-V)^a  W_{(1,0)}^{(2)}  +(1-\xi_p) (1-V)^a  W_{(0,1)}^{(2)},  \\[1em]
    W_{(0,1)} \overset{\cL}= V^a  W_{(0,1)}^{(1)} + \xi_p (1-V)^a  W_{(0,1)}^{(2)}  +(1-\xi_p )(1-V)^a  W_{(1,0)}^{(2)},
     \end{cases}
 \end{equation}
 where
 \begin{itemize}
     \item $V$ is a uniformly distributed random variable on $[0,1]$;
     \item $\xi_p$ is a Bernoulli distributed random variable with parameter $p$;
     \item the $W_{(1,0)}^{(k)}$ and the $W_{(0,1)}^{(k)}$ are respective copies of $W_{(1,0)}$ and $W_{(0,1)}$; 
      \item all the variables are independent. 
 \end{itemize}
 \end{thm}

\begin{rk}\label{rk:loiLX}
The variables $W_{(1,0)}$ and $W_{(0,1)}$ satisfy $W_{(1,0)} \overset{a.s.}= -W_{(0,1)}$. Indeed, by Theorem~\ref{thm:Janson-dim1}, the urn process started from $(1,0)$ satisfies
\begin{equation*}
    \lim_{n\to\infty} \frac{R_n - B_n}{n^a} = W_{(1,0)} \quad\text{and}\quad
    \lim_{n\to\infty} \frac{B_n - R_n}{n^a} = -W_{(1,0)}\quad\text{a.s.}
\end{equation*}
The two colors of the urn are exchangeable, hence the random variable $B_n-R_n$ can be seen as $R_n-B_n$ when started from $(0,1)$, and therefore satisfies 
     \begin{equation*}  
    \lim_{n\to\infty} \frac{B_n - R_n}{n^a} =W_{(0,1)}
    \quad \text{a.s.}
\end{equation*}
As a consequence and using \eqref{eq:L-urn}, we have $L_q \overset{\cL}= (2\xi_q -1)W_{(1,0)}$, and in particular $W_{(1,0)} \overset{\cL}=L_1$.
\end{rk}

Due to the above remark, we may only focus our attention on $W_{(1,0)}$. In particular, the system \eqref{eq:fixed-point-XY} can be reduced to a single equation, as summarized in Theorem~\ref{thm:fixed-point-dim1}. Theorem~\ref{thm:fixed-point-dim1} is then proved.

\medskip
We now give the proof of Theorem~\ref{thm:dim1_solution_dist_eq} on existence and uniqueness of the solution with a finite second-order moment to the fixed-point Equation~\eqref{eq:fixed-point-X-only} when the first moment is fixed.

\begin{proof}[Proof of Theorem~\ref{thm:dim1_solution_dist_eq}]
This proof is greatly inspired by \cite[Lem.~6 and Thm~7]{Chauvin2015}. 
As pointed out in \cite{Chauvin2015}, it could be deduced from the general result in Neininger-R\"uschendorf \cite{NeRu-04}, but we choose to give instead a direct proof via the contraction method.

Let us fix $m\in \dR$, and define $\cP_2(m)$ be the space of probability distributions on $\dR$ that have $m$ as expectation and finite second moment. We consider the Wasserstein distance  $d_\cW$ on $\cP_2(m)$. For a brief overview on this distance, we refer the reader to \cite[Sec.~4.1]{Chauvin2015}, and we recall that 
\begin{equation*}
    d_\cW(\mu,\nu) = \min_{(X,Y)} \sqrt{\dE[(X-Y)^2]},
\end{equation*}
for $(X,Y)$ a random vector with marginal distributions $(\mu,\nu)$. We then introduce
\begin{align*}
    H :\  \cP_2(m) &\longrightarrow \cP_2(m) \\
     \mu &\longmapsto \cL\PAR{V^aW^{(1)}+(2\xi_p-1)(1-V)^aW^{(2)}},
\end{align*}
where $W^{(1)},W^{(2)}$ have distribution $\mu$, $V$ is uniform on $(0,1)$, $\xi_p$ is Bernoulli $\cB(p)$ and all the random variables are independent from each other.

For $\mu\in\cP_2(m)$, it is immediate that $\int xH(\mu)(\dd x) = m$. We want to show that $H$ is Lipschitz-continuous for the Wasserstein distance. Let $\mu_1,\mu_2\in\cP_2(m)$, $(W_1^{(1)},W_2^{(1)})$ and $(W_1^{(2)},W_2^{(2)})$ two independent couples with marginal distributions $(\mu_1,\mu_2)$. Then 
\begin{align*}
d_\cW(H\mu_1,H\mu_2)^2 & \leq \dE\SBRA{\PAR{V^a(W_1^{(1)}-W_2^{(1)})+(2\xi_p-1)(1-V)^a(W_1^{(2)}-W_2^{(2)})}^2} \\
& \leq \dE[V^{2a}]\dV\SBRA{W_1^{(1)}-W_2^{(1)}}+\dE[(1-V)^{2a}]\dV\SBRA{W_1^{(2)}-W_2^{(2)}} \\
& \leq \frac{2}{1+2a} \dE\SBRA{\PAR{W_1^{(1)}-W_2^{(1)}}^2}.
\end{align*}
Thus taking the minimum over all couples $(W_1^{(1)},W_2^{(1)})$ with marginal distributions $(\mu_1,\mu_2)$, we deduce 
\begin{equation*}
d_\cW(H\mu_1,H\mu_2)^2  \leq  \frac{2}{1+2a} d_\cW(\mu_1,\mu_2)^2.
\end{equation*}
Hence $H$ is $\sqrt{\frac{2}{1+2a}}$-Lipschitz, which implies that $H$ is a contraction ($a>1/2$) and that it has unique fixed point in $\cP_2(m)$. Theorem~\ref{thm:dim1_solution_dist_eq} is written in the specific case $m=\frac{1}{\Gamma(1+a)}$.
\end{proof}

\begin{rk}
Equation \eqref{eq:fixed-point-X-only} turns out to be a particular case of fixed point equations studied in \cite{Iksanov-Meiners2015}. More specifically, in the notation of \cite{Iksanov-Meiners2015}, the dimension should be taken to $1$, $N=2$, $T_1=V^{a}\in(0,1)$ and $T_2=(2\xi_p-1)(1-V)^{a}\in(-1,1)$, with $V$ a uniform random variable on $(0,1)$ and $\xi_p$ a Bernoulli variable with parameter $p$, independent of $V$. Since $\dE[T_1^\alpha+|T_2|^\alpha]=2/(1+\alpha a)$, Assumptions (A1)--(A4) of \cite{Iksanov-Meiners2015} are satisfied for $\alpha=1/a\in(1,2)$. 
    
    By \cite[Eq.~(2.31)]{Iksanov-Meiners2015}, we deduce that there exists $c\in\dR$ such that 
    \begin{equation*}
     L_1\overset{\cL}{=}c Z+\mathtt{W}^{a}Y_{1/a},
    \end{equation*}
    where \begin{itemize}
    \item $Z$ is the limit value of a martingale, a.s.\ and in $L^{\beta}$, for $1<\beta<1/a$;
    \item $\mathtt{W}$ is a special endogenous non-negative solution to the tilted equation
    \begin{equation}\label{eq:tilted}
    \mathtt{W}\overset{\cL}{=}V \mathtt{W}^{(1)}+(1-V) \mathtt{W}^{(2)}, 
    \end{equation}
    where $(\mathtt{W}^{(1)},\mathtt{W}^{(2)})$ are independent copies of $\mathtt{W}$, and independent of $V$;
    \item $Y_{1/a}$ is a strictly $1/a$-stable random variable independent of $Z$ and $\mathtt{W}$. 
    \end{itemize}
  We notice that if $\mathtt{W}$ is a solution to \eqref{eq:tilted}, then
    \begin{equation*}
    \dE[\mathtt{W}^2]=2\dE[V^2]\dE[\mathtt{W}^2]+2\dE[V(1-V)]\dE[\mathtt{W}]^2\\
    =\frac{2}{3}\dE[\mathtt{W}^2]+\frac{1}{3}\dE[\mathtt{W}]^2.
    \end{equation*}
    It implies $\dE[\mathtt{W}^2]=\dE[\mathtt{W}]^2$, and we deduce that $\mathtt{W}$ is a constant.

    Recall that $\alpha$-stable distributions don't have second-order moments, unless $\alpha=2$. Under this assumption, we have $\mathtt{W}=0$ since $1/2<a<1$. We thus deduce that up to a constant, $Z$ has the same distribution than $L_1$.
    \end{rk}

\subsection{Existence of a density}\label{sec:existence-density-dim1}

In this section we will prove Theorem~\ref{thm:density_dim_1}. Let $W:=W_{(1,0)}$ be the first elementary limit distribution, as defined in Theorem~\ref{thm:Janson-dim1}, and $\varphi_W$ be its Fourier transform, defined for $t\in\dR$ by
 \begin{equation*}
     \varphi_W(t) := \dE\SBRA{\ee^{itW}}.
 \end{equation*}
We recall that $W\overset{\cL}{=}L_1$ satisfies the fixed-point Equation~\eqref{eq:fixed-point-X-only}. 
 \begin{thm}
 \label{thm:fct-carac-X}
     Let $W$ be the first elementary limit distribution of a large two-color Pólya urn process with random replacement matrix $A$ and ratio $a\in\PAR{1/2,1}$. Then 
     \begin{itemize}
         \item the support of $W$ is the whole real line $\dR$;
         \item for any $\rho\in\PAR{0,\frac{1}{a}}$, and any $k\geq 1$, there exists a constant $C>0$ such that for any $t\in\dR^*$, 
         \begin{equation}
        \label{eq:upper_bound_Fourier}
             \ABS{\varphi_X(t)} \leq \frac{C}{\ABS{t}^{k\rho}}.
         \end{equation} 
     \end{itemize}
 \end{thm}
Before proving the theorem, we mention a few obvious consequences. 

\begin{cor}
\label{cor:density_W_dim_1}
   For any initial composition $\alpha=(1,0)$ or $(0,1)$, the random variable $W_\alpha$ admits a positive, bounded and smooth density. Its support is the whole real line $\dR$. 
\end{cor}

 \begin{proof}[Proof of Corollary~\ref{cor:density_W_dim_1}]
     Recall that $W_{(1,0)} \overset{a.s.}= -W_{(0,1)}$. We notice that $1<\frac{1}{a}<2$, then taking $k=1$ in Theorem \ref{thm:fct-carac-X}, we deduce that the characteristic function $\varphi_W$ of $W$ is integrable on $\dR$. By Fourier inversion theorem, $W$ has a bounded density on $\dR$. Since the support of $W$ is $\dR$, its density is positive a.e. Besides, for any $k\geq 1$, $t\mapsto t^k\varphi_W(t)$ is integrable on $\dR$, which implies that the density function is of class $\cC^k$. The corollary is proved.
 \end{proof}
 
\begin{rk}
From the fixed-point equation \eqref{eq:fixed-point-X-only}, we deduce that the density function $f$ of $L_1$ satisfies the following integral equation, for all $x\in \dR$,
\begin{equation*}
    f(x) = \int_{\dR}\int_{0}^1 \frac{1}{z^a} f\big(x-(1-z)^ay\big)\big(pf(y)+(1-p)f(-y)\big)\mathrm{d}z\mathrm{d}y.
\end{equation*}
This equation is unfortunately too complex to easily deduce properties on the distribution of $L_1$.
\end{rk}

We recall that $L_q\overset{\cL}{=}(2\xi_q-1)W$, where $\xi_q$ is a Bernoulli random variable with parameter $q$, independent on $W$. Then we easily deduce that Theorem~\ref{thm:density_dim_1} is a direct consequence of Corollary \ref{cor:density_W_dim_1}. 
See Figure~\ref{fig:simulations} in Section~\ref{sec:approx-density} for a few illustrations of the density.

 \begin{proof}[Proof of Theorem \ref{thm:fct-carac-X}]
 We proceed as in the proof of \cite[Thm~6]{Chauvin2015}; the details are slightly different but still very similar to their case. Indeed, we only add one ball at each step but we do it randomly, whereas they insisted that they needed to add at least two balls at each step since they do it deterministically. We have been inspired by \cite[Lem.~3.1 and Lem.~3.2]{Liu2001} as well.

We start by proving that $\text{Supp}(W)=\dR$, with $\text{Supp}(W)$ denoting the support of $W$. By Theorem~\ref{thm:Janson-dim1}, the variance of $W$ is non-zero, so there exist $w_1,w_2\in\text{Supp}(W)$ with $w_1\neq w_2$. Because of the fixed-point equation \eqref{eq:fixed-point-X-only}, for any $t\in[0,1]$ we have
 \begin{equation}
 \label{eq:supp-X}
     w_1 t^a  +w_2(1-t)^a \in\text{Supp}(W),
 \end{equation}
 since $\dP(\xi_p=1)=p>0$. 
 Moreover, it holds that $[w_1,w_2]\subset \BRA{w_1 t^a  +w_2(1-t)^a :\ t\in[0,1]}$. This entails that $[w_1,w_2]\subset\text{Supp}(W)$. 
 
 Using $\dP(\xi_p=-1)=1-p>0$, similar computations lead to $[w_1,-w_2]\subset\text{Supp}(W)$. Moreover, since $\dE[W]>0$, there exists $w>0$ in $\text{Supp}(W)$. Taking $-w_1=w_2=w>0$, we obtain that $[-w,w]\subset\text{Supp}(W)$. 

Next, choosing $t=1/2$, Equation \eqref{eq:supp-X} ensures that as soon as $w\in\text{Supp}(W)$, $2^{1-a }t\in\text{Supp}(W)$. Since $a <1$, $2^{1-a }>1$ and the images of $[-w,w]$ by the iterates of the homothetic transformation $w \mapsto 2^{1-a }w$ fill the whole real line, which yields that $\text{Supp}(W)=\dR$.

\medskip

We move to the proof of \eqref{eq:upper_bound_Fourier}, and we split the argument in four steps.

\medskip

\textit{Step 1:} We start by showing that for any $t\neq0$, $\ABS{\varphi_W(t)}<1$. Assume there exists $t_0\in\dR$ such that $\ABS{\varphi_W(t_0)}=1$, then  $\varphi_W(t_0)=\ee^{i\theta_0}$ for some $\theta_0\in\dR$. We deduce that $\dE\SBRA{\mathrm{Re}\PAR{1-e^{i(t_0W-\theta_0)}}}=0$, and then $\ee^{it_0W}=\ee^{i\theta_0}$ a.s. This is possible only if $t_0=0$, since $\text{Supp}(W)=\dR$.

\medskip

\textit{Step 2:} We show that
\begin{equation}
\label{eq:0-1-law}
    \ds{\limsup_{t\to\pm\infty} \ABS{\varphi_W(t)}\in\BRA{0,1}}.
\end{equation}
Using again the fixed-point equation \eqref{eq:fixed-point-X-only}, we obtain by conditioning on $V$ that for any $t\in\dR$,
 \begin{equation}
 \label{eq:fct-carac-X-V}
     \varphi_W(t)=p\dE\SBRA{\varphi_W(tV^a )\varphi_W(t(1-V)^a )}+(1-p)\dE\SBRA{\varphi_W(tV^a )\varphi_W(-t(1-V)^a )}.
 \end{equation}
In particular, since $V$ is uniformly distributed on $[0,1]$, it is a.s.\ non-zero. Consequently, Equation~\eqref{eq:fct-carac-X-V} together with Fatou's lemma ensure that
 \begin{align*}
     \limsup_{t\to\pm\infty} \ABS{\varphi_W(t)} &\leq 
    p\limsup_{t\to\pm\infty} \dE\SBRA{\ABS{\varphi_W(tV^a )}\times\ABS{\varphi_W(t(1-V)^a )}} \\
    &\qquad\qquad +(1-p)\limsup_{t\to\pm\infty}\dE\SBRA{\ABS{\varphi_W(tV^a )}\times\ABS{\varphi_W(-t(1-V)^a )}}\\
    &\leq 
    p\limsup_{t\to\pm\infty} \ABS{\varphi_W(t)}^2  +(1-p)\limsup_{t\to\pm\infty}\ABS{\varphi_W(t)}.
 \end{align*}
 We deduce 
 \begin{equation*}
     \limsup_{t\to\pm\infty} \ABS{\varphi_W(t)} 
    \leq 
    \limsup_{t\to\pm\infty} \ABS{\varphi_W(t)}^2\leq \PAR{\limsup_{t\to\pm\infty} \ABS{\varphi_W(t)}}^2,
 \end{equation*}
 by monotony of $x\mapsto x^2$ on $\dR_+$, concluding the proof of \eqref{eq:0-1-law}.

\medskip

\textit{Step 3:} We now show that in \eqref{eq:0-1-law} the only possible value is $0$. Reasoning by contradiction, we assume that $\ds{\limsup \ABS{\varphi_W(t)}=1}$ as $t\to\pm\infty$. Let $t_0>0$ and $\varepsilon>0$ be such that $0<\ABS{\varphi_W(t_0)}<1-\varepsilon$. Choose $t_1 = t_1(\varepsilon)$ and $t_2 = t_2(\varepsilon)$ such that $0<t_1<t_0<t_2<\infty$ and
\begin{equation*}
    \ABS{\varphi_W(t_1)}=\ABS{\varphi_W(t_2)}=1-\varepsilon \quad\text{and}\quad
    \ABS{\varphi_W(t)}\leq1-\varepsilon\ \text{for $t\in[t_1,t_2]$.}
\end{equation*}
This is possible thanks to the intermediate value theorem, since $\varphi_W$ is continuous and we assumed that $\limsup_{t\to\pm\infty} \ABS{\varphi_W(t)}=1=\ABS{\varphi_W(0)}$. Moreover, if $(\varepsilon_n)_{n\geq 1}$ is a sequence tending to $0$ as $n$ goes to infinity and if $y$ is a limit point of ($t_1(\varepsilon_n))_{n\geq 1}$, then by definition of $t_1$ and the continuity of $\varphi_W$, $\ABS{\varphi_W(y)} = 1$. This entails that $y=0$ and that $\lim_{\varepsilon\to0} t_1(\varepsilon)=0$.

Again, Equation \eqref{eq:fct-carac-X-V} implies that for all $t>0$,
 \begin{equation*}
    \ABS{\varphi_W(t)} \leq 
    p\dE\SBRA{\ABS{\varphi_W(tV^a )}\times \ABS{\varphi_W(t(1-V)^a )}} +(1-p)\dE\SBRA{\ABS{\varphi_W(tV^a )}\times \ABS{\varphi_W(-t(1-V)^a )}},
 \end{equation*}
 with $\ABS{\varphi_W(-t(1-V)^a )}=\ABS{\overline{\varphi_W(t(1-V)^a )}}=\ABS{\varphi_W(t(1-V)^a )} $, and, for all $n\geq 1$ and $t>0$,
 \begin{align}
 \label{eq:majo-phi-V}
    \ABS{\varphi_W(t)} &\leq \dE\SBRA{\ABS{\varphi_W(tV^a )}\times \ABS{\varphi_W(t(1-V)^a )}}\\
    &\leq \dE\SBRA{\ABS{\varphi_W(tM)}} \leq \phi_n(t) ,\nonumber
 \end{align}
where $M =V^a$ and $\phi_n(t)=\dE\SBRA{\ABS{\varphi_W(tM_1\cdots M_n)}}$, with $M_1,\ldots,M_n$ i.i.d.\ copies of $M$. (Notice that one could also take $M = (1-V)^a $ or $M = \max(V^a ,(1-V)^a )$.) On the one hand, using the definitions of $t_1,t_2$, we find
 \begin{align*}
     \phi_n(t) &\leq (1-\varepsilon)\dP(t_1<tM_1\cdots M_n\leq t_2) +1 - \dP(t_1<tM_1\cdots M_n\leq t_2)\\
     &\leq 1-\varepsilon\dP(t_1<tM_1\cdots M_n\leq t_2)
 \end{align*}
 and by \eqref{eq:majo-phi-V}
  \begin{equation*}
    1-\varepsilon = \ABS{\varphi_W(t_2)}\leq\dE[\phi_n(t_2V^a )\phi_n(t_2(1-V)^a )],
 \end{equation*}
 which leads to 
 \begin{equation}
 \label{eq:somme-proba}
     \dP(t_1<t_2V^a  M_1\cdots M_n\leq t_2)+\dP(t_1<t_2(1-V)^a  M_1\cdots M_n\leq t_2)\leq 1 +\varepsilon.
 \end{equation}
 Besides, when $\varepsilon$ goes to $0$, we find that 
\begin{equation*}
    \dP(t_1<t_2x M_1\cdots M_n\leq t_2) \longrightarrow \dP(xM_1\cdots M_n\leq 1),
\end{equation*}
which, together with dominated convergence, ensures that as $\epsilon\to 0$
\begin{equation*}
    \dP(t_1<t_2 V^a  M_1\cdots M_n\leq t_2) \longrightarrow  \dP(V^a  M_1\cdots M_n\leq 1)
\end{equation*}
and the same thing holds for $(1-V)^a $.  Thus, by \eqref{eq:somme-proba} and taking $\varepsilon\to 0$, we deduce
\begin{equation}\label{eq:somme-proba2}
\dP(V^a  M_1\cdots M_n\leq 1)+\dP((1-V)^a  M_1\cdots M_n\leq 1)\leq 1.
\end{equation}
On the other hand, we  also have
\begin{equation*}
    \dP(xM_1\cdots M_n\leq 1) \underset{n\to \infty}{\longrightarrow} 1
\end{equation*}
thanks to Markov's inequality, since $\dE[M^{1/a }]<1$. Using again dominated convergence, we find 
\begin{equation*}
     \dP(V^a  M_1\cdots M_n\leq 1)+\dP((1-V)^a  M_1\cdots M_n\leq 1) \underset{n\to \infty}{\longrightarrow} 2
\end{equation*}
and we have a contradiction with \eqref{eq:somme-proba2}. The third step is proved.

\medskip

\textit{Step 4:} Finally, we focus our attention on the integrability of $\varphi_W$ and prove \eqref{eq:upper_bound_Fourier}. Let $\varepsilon>0$ and $T_\varepsilon>0$ be such that $\ABS{\varphi_W(t)}\leq\varepsilon$ as soon as $\ABS{t}\geq T_\varepsilon$. The quantity $T_\varepsilon>0$ exists (and is finite)\ due to the third step of the proof. We have from Equation~\eqref{eq:fct-carac-X-V} that for any $t\in\dR$,
 \begin{align*}
     \ABS{\varphi_W(t)} &\leq p\dE\SBRA{\ABS{\varphi_W(tV^a )}\times \ABS{\varphi_W(t(1-V)^a )}} +(1-p)\dE\SBRA{\ABS{\varphi_W(tV^a )}\times \ABS{\varphi_W(-t(1-V)^a )}} \\
     &\leq \varepsilon\dE\SBRA{\ABS{\varphi_W(tV^a )}}+\dP((1-V)^a \ABS{t}\leq T_\varepsilon).
 \end{align*}

 Hence, since $\dE[V^{-a \rho}]<\infty$ as soon as $\rho\in\PAR{0,1/a }$, we find thanks to Markov's inequality that there exists a constant $C=\frac{1}{1-\rho a}>0$ such that, for any $t\neq0$,
 \begin{equation*}
     \ABS{\varphi_W(t)} \leq \varepsilon\dE\SBRA{\ABS{\varphi_W(tV^a )}}+C\PAR{\frac{T_\varepsilon}{\ABS{t}}}^{\rho}.
 \end{equation*}
Then, by iterating the last inequality, we find that
\begin{equation*}
    \ABS{\varphi_W(t)} \leq \varepsilon^n\dE\SBRA{\ABS{\varphi_W(tV_1^a \ldots V_n^a )}}+C\PAR{\frac{T_\varepsilon}{\ABS{t}}}^{\rho}\sum_{k=0}^{n-1}\PAR{\varepsilon\dE[V^{-a \rho}]}^k.
\end{equation*}
When $n$ goes to infinity, this implies that
 \begin{equation}
 \label{eq:majo-varphiW-dim1}
     \ABS{\varphi_W(t)} \leq C\PAR{\frac{T_\varepsilon}{\ABS{t}}}^{\rho}\frac{1}{1- \varepsilon\dE[V^{-a \rho}]}=\frac{1}{1-a\rho-\varepsilon}\PAR{\frac{T_\varepsilon}{|t|}}^\rho
 \end{equation}
 if $\varepsilon$ is chosen such that $\varepsilon\dE[V^{-a \rho}]<1$.
This is enough to conclude that for any $\rho\in(0,1/a)$, we have $\varphi_X(t)\underset{\pm\infty}{=} O(t^{-\rho})$.
We implement \eqref{eq:majo-varphiW-dim1} in Equation \eqref{eq:fct-carac-X-V}, and we thus  obtain, for any $\rho\in\PAR{0,1/a}$, $\varepsilon\in \PAR{0,1-a\rho}$, and $\ABS{t}>T_\varepsilon$, 
\begin{align*}
\ABS{\varphi_W(t)} 
     &\leq \varepsilon\dE\SBRA{\ABS{\varphi_W(tV^a )}}+\dE\SBRA{\ABS{\varphi_W(V^{a}t)}\ind_{(1-V)^a \ABS{t}\leq T_\varepsilon}}\\
     &\leq \varepsilon\dE\SBRA{\ABS{\varphi_W(tV^a )}}+\frac{1}{1- a\rho-\varepsilon}\PAR{\frac{T_\varepsilon}{\ABS{t}}}^{\rho}\dE\SBRA{V^{-a\rho}\ind_{V\geq 1-\PAR{\frac{T_\varepsilon}{\ABS{t}}}^{1/a}}}\\
     &\leq \varepsilon\dE\SBRA{\ABS{\varphi_W(tV^a )}}+\frac{1}{1- a\rho-\varepsilon}\PAR{\frac{T_\varepsilon}{\ABS{t}}}^{\rho}\dE\SBRA{V^{-1}\ind_{V\geq 1-\PAR{\frac{T_\varepsilon}{\ABS{t}}}^{\rho}}}\\
     &\leq \varepsilon\dE\SBRA{\ABS{\varphi_W(tV^a )}}-\frac{1}{1- a\rho-\varepsilon}\PAR{\frac{T_\varepsilon}{\ABS{t}}}^{\rho}\ln\PAR{1-\PAR{\frac{T_\varepsilon}{\ABS{t}}}^{\rho}}.
\end{align*}
As $-\ln(1-u)\leq 2 u$ for all $ u\in(0,1/2)$, there exists a constant $C_\varepsilon$ (that may change from line to line) such that for all $ \ABS{t}>2^{1/\rho}T_\varepsilon$,
\begin{equation*}
\ABS{\varphi_W(t)} 
     \leq \varepsilon\dE\SBRA{\ABS{\varphi_W(tV^a )}}+C_\varepsilon\PAR{\frac{T_\varepsilon}{\ABS{t}}}^{2\rho},
\end{equation*}
and then
\begin{equation*}
\ABS{\varphi_W(t)} 
     \leq C_\varepsilon\PAR{\frac{T_\varepsilon}{\ABS{t}}}^{2\rho}.
\end{equation*}
Iterating this idea, we obtain that for any $k\geq 1$, $\ABS{\varphi_W(t)}=O\bigl(\ABS{t}^{-k\rho}\bigr)$ as $\ABS{t}\to\infty$.
 \end{proof}

\subsection{Moments  of the superdiffusive distribution}\label{sec:moments-dim1}

Recall that $W=W_{(1,0)}$, where $W_{(1,0)}$ has been defined in Theorem~\ref{thm:Janson-dim1}. The main objective of this part is to prove the following result, as well as Theorem~\ref{thm:moments}. 
\begin{thm}\label{thm:exp-moments}
The random variables $L_q$ and $W$ have finite moments of all order and their probability distribution are determined by their moments.
Moreover, their Laplace series have an infinite radius of convergence. 
\end{thm}

To prove these results, we will use the following Carleman's criterion (see e.g.\ \cite{Simon98}).
\begin{thm}[Carleman's criterion]\label{thm:carleman}
Let $X$ be a real random variable with finite moments of all order. If $\sum_{k\geq 1}\PAR{\dE[|X|^{2k}]}^{-1/2k}=\infty$, then the distribution of $X$ is uniquely determined by its moments.
\end{thm}

It is already known that $L_q$ has finite moments of all order, see \cite[Eq.~(4.12)]{BercuHG2019}. By Remark~\ref{rk:loiLX}, we have $\dE[\ABS{L_q}^k]=\dE[\ABS{W}^k]$ for all $k\geq 1$.
The idea is to use the fixed-point equation \eqref{eq:fixed-point-X-only} satisfied by $W$ to prove that $W$ (and thus $L_q$) satisfies the assumption of Carleman's criterion. We first prove a preliminary estimate.

\begin{lem}
\label{lem:GammaMomentsbound}
There exists a constant $C>0$ such that for any $k\geq1$,
\begin{equation}
\label{eq:induction-gamma}
\frac{\dE\SBRA{|W|^k}}{k!}\leq \frac{C^k}{\Gamma(ka+1)}\PAR{\frac{2}{a}}^{k-1}.
\end{equation}
\end{lem}
\begin{proof}
We prove the proposition by induction. 
The second moment of $W\overset{\cL}{=}L_1$ is given in \cite{Janson2004,Bercu2018} and recalled in this paper in Theorem~\ref{thm:Janson-dim1}.
Let $C=\Gamma(a+1)\dE\SBRA{L_1^2}^{1/2}$. For $k=1$ and $k=2$, we easily see that Equation \eqref{eq:induction-gamma}  holds (thanks to Cauchy-Schwarz inequality for $k=1$). 

Let $k\geq 3$ and assume that \eqref{eq:induction-gamma} is true for all $j\in\BRA{1,\ldots,k-1}$. We now study the $k$-th moment of $W$.
 Taking $(W^{(1)},W^{(2)})$ a couple of independent copies of $W$, by  \eqref{eq:fixed-point-X-only} and independence between $V$ and $(W^{(1)},W^{(2)})$, we have
\begin{align*}
\dE\SBRA{\ABS{W}^k}&\leq \dE\SBRA{\PAR{V^{a}\ABS{W^{(1)}}+(1-V)^{a}\ABS{W^{(2)}}}^k}\\
&\leq 2\dE[V^{ka}]\dE\SBRA{\ABS{W}^k}+\sum_{j=1}^{k-1}\binom{k}{j}\dE\SBRA{V^{ja}(1-V)^{(k-j)a}}\dE\SBRA{\ABS{W}^j}\dE\SBRA{\ABS{W}^{k-j}}.
\end{align*}
As $\dE[V^{ka}]=\frac{1}{1+ka}$, we deduce 
\begin{equation*}
\dE\SBRA{\ABS{W}^k}\leq \frac{ka+1}{ka-1}\sum_{j=1}^{k-1}\binom{k}{j}\dE\SBRA{V^{ja}(1-V)^{(k-j)a}}\dE\SBRA{\ABS{W}^j}\dE\SBRA{\ABS{W}^{k-j}},
\end{equation*}
which by induction immediately leads to
\begin{align*}
\frac{\dE\SBRA{\ABS{W}^k}}{k!} &\leq \frac{ka+1}{ka-1}\sum_{j=1}^{k-1}\frac{\Gamma(ja+1)\Gamma((k-j)a+1)}{\Gamma(ka+2)}\frac{\dE\SBRA{\ABS{W}^j}}{j!}\frac{\dE\SBRA{\ABS{W}^{k-j}}}{(k-j)!}\\
&\leq \frac{k-1}{ka-1}\frac{a}{2}\,\PAR{\frac{2}{a}}^{k-1}\frac{C^k}{\Gamma(ka+1)}\\
&\leq \PAR{\frac{2}{a}}^{k-1}\frac{C^k}{\Gamma(ka+1)},
\end{align*}
since for $k\geq 3$, $a(k-1)\leq 2ka-2$. 
 The lemma is proved.
\end{proof}

\begin{proof}[Proof of Theorem~\ref{thm:exp-moments}]
By Remark \ref{rk:loiLX}, if we have the result for $W$, we immediately obtain it for $L_q$.
By Lemma \ref{lem:GammaMomentsbound}, there exists $C>0$ such that for all $k\geq 1$,
\begin{equation*} 
\dE\SBRA{\ABS{W}^{k}}^{-1/k}
\geq 
\frac{a}{2C}
\PAR{\frac{2}{a}}^{1/k}\PAR{\frac{k!}{\Gamma(ka+1)}}^{-1/k}.
\end{equation*}
Recalling Stirling's formula $\Gamma(z)\underset{z\to +\infty}{\sim} \sqrt{\frac{2\pi}{z}}\PAR{\frac{z}{\ee}}^z$, we deduce 
\begin{equation*}
    \PAR{\frac{2}{a}}^{1/k}\PAR{\frac{k!}{\Gamma(ka+1)}}^{-1/k}
    \underset{k\to+\infty}{\sim} \frac{a^a \ee^{1-a}}{k^{1-a}}.
\end{equation*}
Since $1-a<1$, $\sum_{k\geq 1}1/k^{1-a}=\infty$, we conclude by Carleman's criterion.
In fact, the estimate of Lemma~\ref{lem:GammaMomentsbound} ensures that the radius of convergence of the Laplace series of $|W|$ is infinite.
\end{proof}

\begin{cor}
\label{cor:laplaceL}
    For any first step parameter $q\in[0,1]$ and memory parameter $p\in(3/4,1)$, the moment-generating functions of $L_q$ and $\vert L_q\vert$ are well defined and finite on $\dR$, meaning that for all $ t\in\dR$,
    \begin{equation*}
    \dE[\ee^{tL_q}]<\infty\quad \text{and} \quad \dE[\ee^{t|L_q|}]<\infty.
    \end{equation*}
\end{cor}
\begin{proof}
By Tonelli's theorem and Theorem~\ref{thm:moments}, for all $t\in\dR$, 
    \begin{equation*}
    \dE\SBRA{\ee^{t\ABS{L_q}}}=\sum_{k=0}^\infty \frac{ \dE[\ABS{tL_q}^k]}{k!} = \sum_{k=0}^\infty \frac{|t|^k\dE\SBRA{|W|^k}}{k!}<\infty. 
    \end{equation*}
    Then $\ds{\dE\SBRA{\ee^{tL_q}}=\sum_{k=0}^\infty \frac{ t^kL_q^k}{k!}}$ is absolutely convergent.
\end{proof}

Bercu noticed in \cite[Rem.~3.2]{Bercu2018} that $L_q$ is a sub-Gaussian random variable. We prove here that $L_q^2$ has finite exponential moments, this indicates that the distribution of $L_q$ has  really thin tails. 
\begin{cor}\label{cor:laplaceL^2}
     For any first step parameter $q\in[0,1]$ and memory parameter $p\in(3/4,1)$, the moment-generating function of $L_q^2$ is well defined and finite on $\dR$. 
\end{cor}
\begin{proof}
    For $t\in\dR$,
    \[
    0\leq \ee^{tL_q^2} = \sum_{k=0}^\infty \frac{t^kL_q^{2k}}{k!}\leq \sum_{k=0}^\infty \frac{|t|^k W^{2k}}{k!}.
    \]
    Hence for all $t\in\dR$, 
    \[
    \dE\SBRA{\ee^{tL_q^2}}\leq \sum_{k=0}^\infty \frac{|t|^k\dE\SBRA{W^{2k}}}{k!}.
    \]
    Hereafter, we have from Lemma \ref{lem:GammaMomentsbound} that
    \[
    \sum_{k=0}^\infty \frac{|t|^k\dE\SBRA{W^{2k}}}{k!} \leq 
    \frac{a}{2}\sum_{k=0}^\infty |t|^k\PAR{\frac{2C}{a}}^{2k}
    \frac{(2k)!}{k!\Gamma(2ka+1)}.
    \]
    Since $a>1/2$, we obtain that \[\frac{k!\Gamma(2ka+1)}{(2k)!}\frac{(2k+2)!}{(k+1)!\Gamma(2ka+2a+1)}\sim \frac{(2k)^2}{k (2ka)^{2a}}\sim \frac{2^{2(1-a)}}{a^{2a}k^{2a-1}} \to 0\]
    when $k$ goes to $\infty$, and thus the radius of convergence is infinite.
\end{proof}

Since we know that $L_1$ has finite moments of all order and its moment-generating function is well defined, we can now prove Theorem~\ref{thm:moments}. However, before giving the proof,  we make some comments on the result.
\begin{rk}
\label{rk:moments-L1}
From the recursive equation in Theorem~\ref{thm:moments}, we easily compute by induction 
    \[
    m_1=1, \quad m_2=\frac{a}{2a-1},\quad m_3=\frac{a+1}{2(2a-1)}\quad \text{and}\quad m_4=\frac{a(2a^2+2a-1)}{(4a-1)(2a-1)^2},
    \]
    and then the successive moments of $L_1\overset{\cL}{=}W$, are given by
    \begin{align*}
    \dE[L_1]&=\frac{1}{\Gamma(a+1)},\quad &\dE[L_1^2]&=\frac{1}{(2a-1)\Gamma(2a)},\quad\\
    \dE[L_1^3]&=\frac{a+1}{a(2a-1)\Gamma(3a)},\quad  &\dE[L_1^4]&=\frac{6(2a^2+2a-1)}{(4a-1)(2a-1)^2\Gamma(4a)}.
    \end{align*}
  We recover the expressions obtained by Bercu in \cite{Bercu2018}. 
\end{rk}

Since $a\mapsto m_2(a)=\frac{a}{2a-1}$ is positive and decreasing on $(1/2,1)$, we have the following  trivial corollary of Theorem~\ref{thm:moments}.
\begin{cor}
The following holds:
\begin{enumerate}
    \item For all $ k\geq 1$, $\dE[L_1^k]$ is positive. 
    \item For all $k\geq 2$, $a\mapsto m_k(a)$ is decreasing on $(1/2,1)$  with 
    \[\lim_{a\to 1/2}m_k(a)=\infty \quad \text{ and }\quad \lim_{a\to 1}m_k(a)=1.
    \]
Consequently, for all $ k\geq 1$, $\displaystyle{\lim_{a\to 1}\dE\bigl[\ABS{L_1-1}^k\bigr]}=0$.
\item Since $L_q\overset{\cL}{=}(2\xi_q-1)L_1$ with $\xi_q\sim\cB(q)$ independent of $L_1$ (see Remark~\ref{rk:loiLX}), we have for all $k\geq 1$ and $t\in\dR$,
\begin{align*}
    \dE[L_q^k]&=\PAR{q+(1-q)(-1)^k}\dE[L_1^k],\\
\dE\SBRA{\ee^{tL_q}}&=\sum_{k\geq 0}\frac{m_k}{\Gamma(ka+1)}\PAR{q+(1-q)(-1)^k}t^k.
\end{align*}
\end{enumerate}
\end{cor}

\begin{proof}[Proof of Theorem \ref{thm:moments}]

Since $L_1$ is solution of the fixed-point equation \eqref{eq:fixed-point-X-only}, using the same notations as in Theorem~\ref{thm:fixed-point-dim1}, by independence between the variables, we have for $\mu_k:= \dE[W^k]$,
\begin{align*}
    \mu_k&=\dE\SBRA{\sum_{j=0}^k\binom{k}{j}(1-V)^{ja}V^{(k-j)a}(2\xi_p-1)^j\PAR{L_1^{(2)}}^j\PAR{L_1^{(1)}}^{k-j}}\\
    &=\sum_{j=0}^k\binom{k}{j}\dE\SBRA{(1-V)^{ja}V^{(k-j)a}}\dE\SBRA{(2\xi_p-1)^j}\mu_j\mu_{k-j}\\
    &=\sum_{j=0}^k\binom{k}{j}\frac{\Gamma(ja+1)\Gamma((k-j)a+1)}{\Gamma(ka+2)}c_j\mu_{j}\mu_{k-j},
\end{align*}
with $c_j=1$ for even $j$ and $c_j=a$ for odd $j$.
Thus
\begin{align*}
    \frac{ka-c_k}{ka+1}\mu_k
    &=\sum_{j=1}^{k-1}\binom{k}{j}\frac{\Gamma(ja+1)\Gamma((k-j)a+1)}{\Gamma(ka+2)}c_j\mu_{j}\mu_{k-j}.
\end{align*}
We introduce $m_k$ such that $\mu_k=\frac{k!}{\Gamma(ka+1)}m_k$. We thus deduce that $m_k$ satisfies the recursive equation: $m_1=1$, and for $k\geq 2$,
\begin{equation*}
(ka-c_k)k!m_k=\sum_{j=1}^{k-1} k!c_jm_{j}m_{k-j},
\end{equation*}
which implies \eqref{eq:recursion_moments}.
Theorem \ref{thm:moments} is proved. \end{proof}

\subsection{Approximation of the density}\label{sec:approx-density}

The aim of this part is to propose a few approximations of the density of $L_q$, as illustrated in Figure~\ref{fig:simulations}.
\begin{figure}[ht!]
    \centering
    \includegraphics[width=5cm]{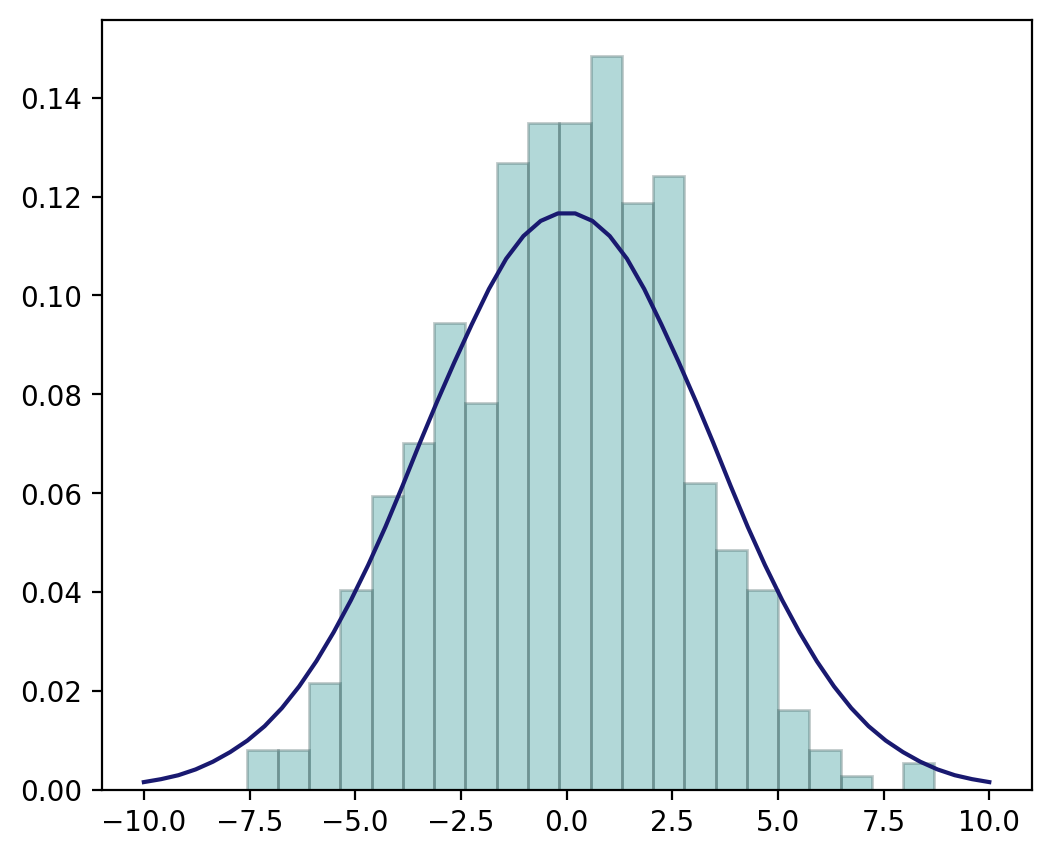}
    \includegraphics[width=5cm]{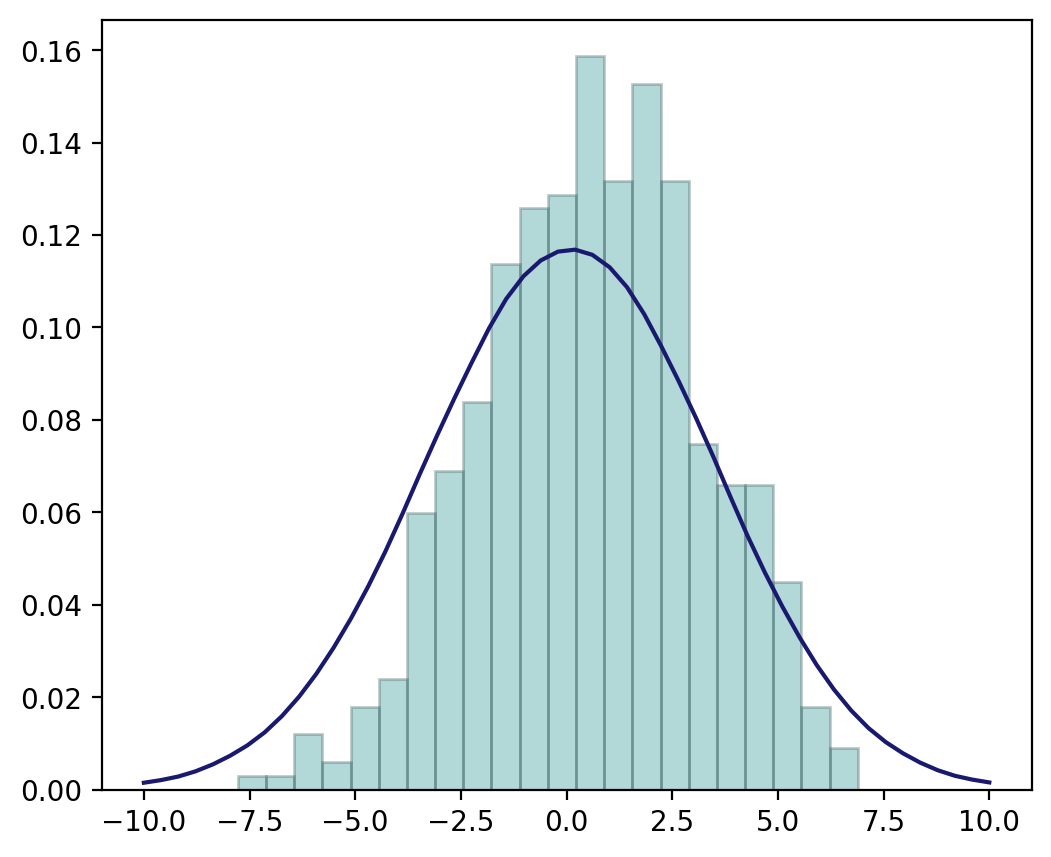}
    \includegraphics[width=5cm]{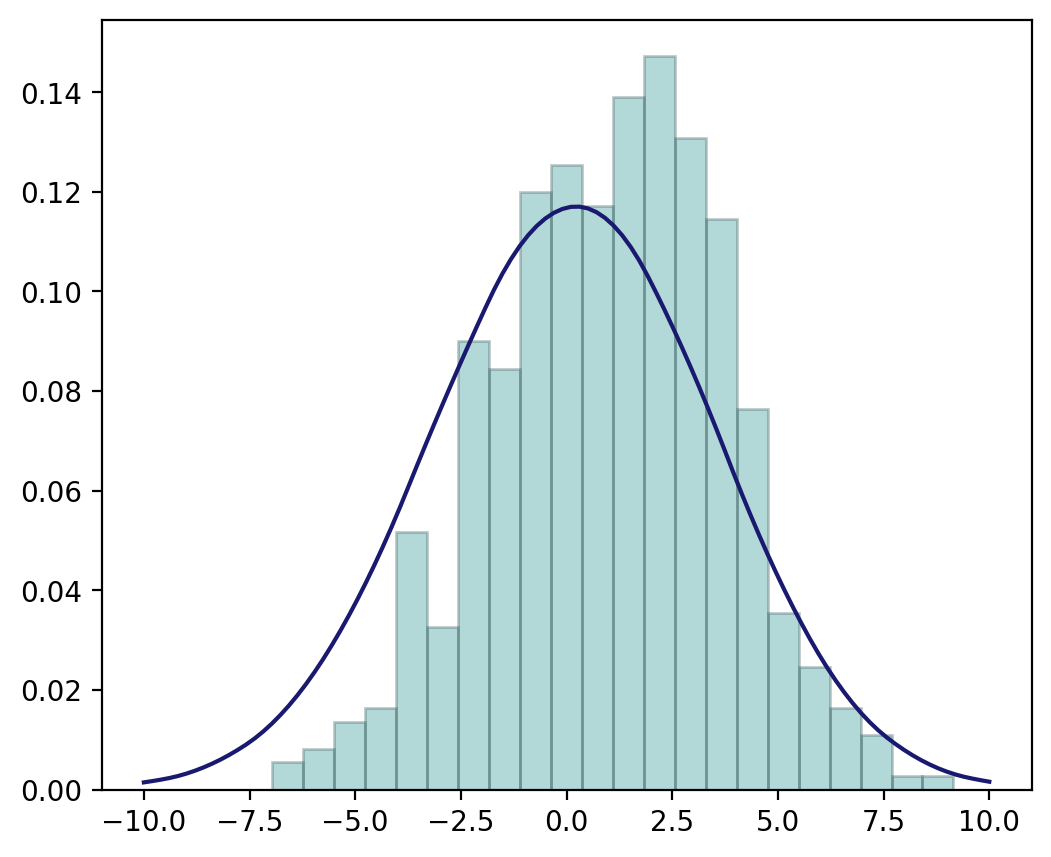}\\\medskip
    \includegraphics[width=5cm]{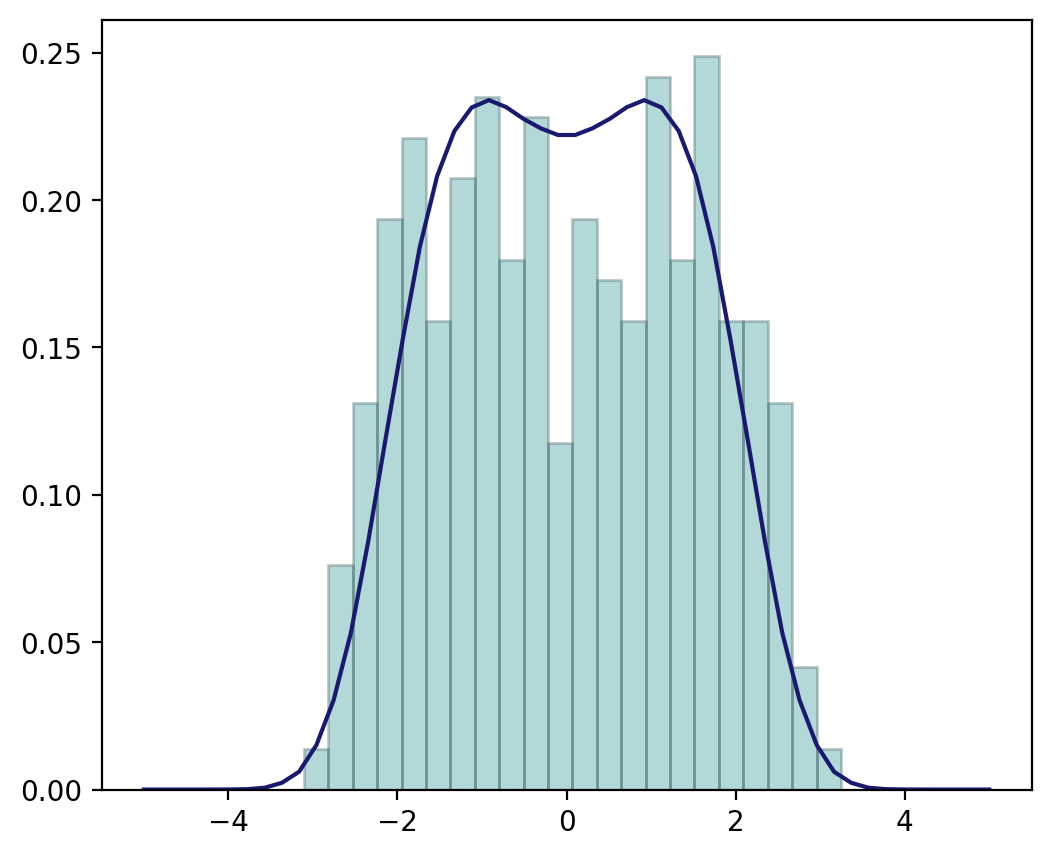}
    \includegraphics[width=5cm]{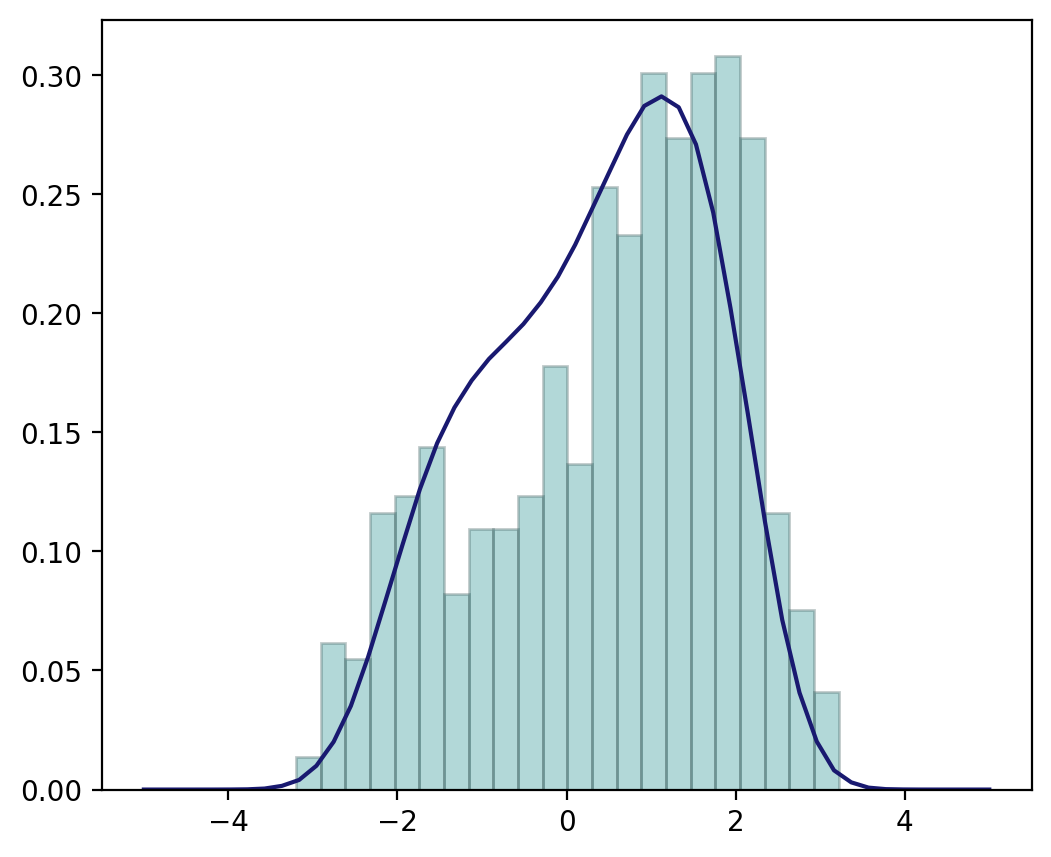}
    \includegraphics[width=5cm]{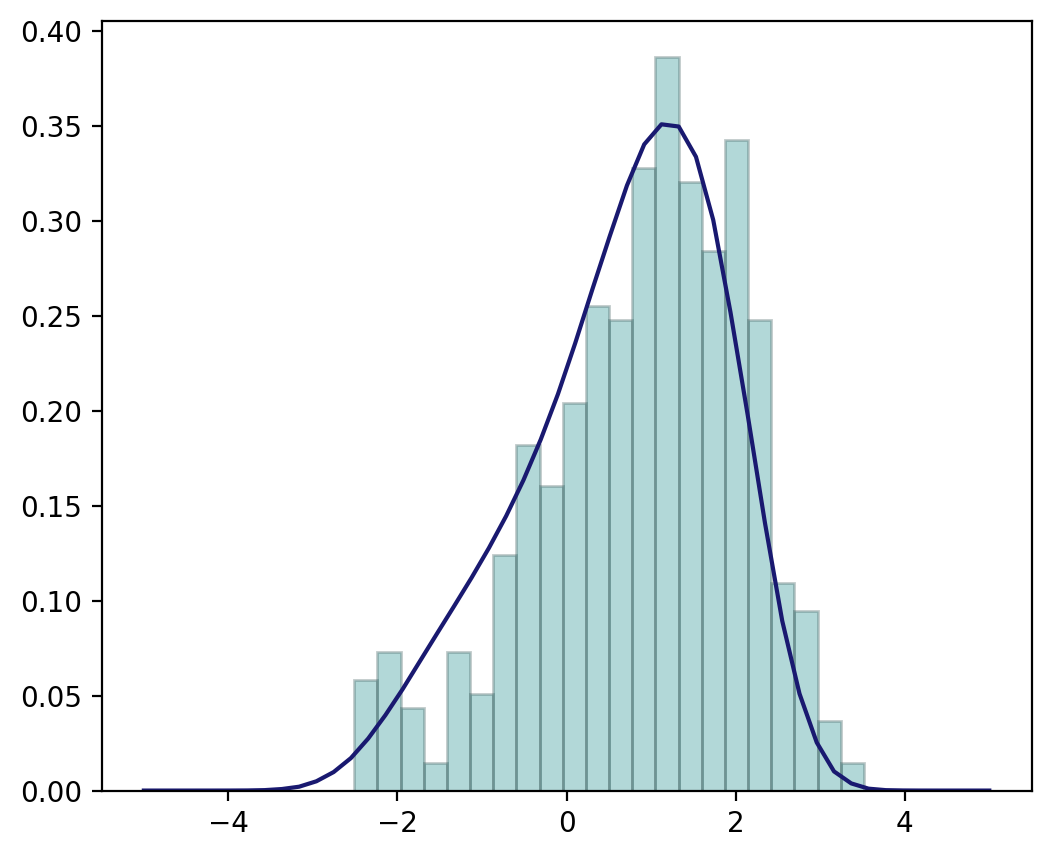}\\\medskip
    \includegraphics[width=5cm]{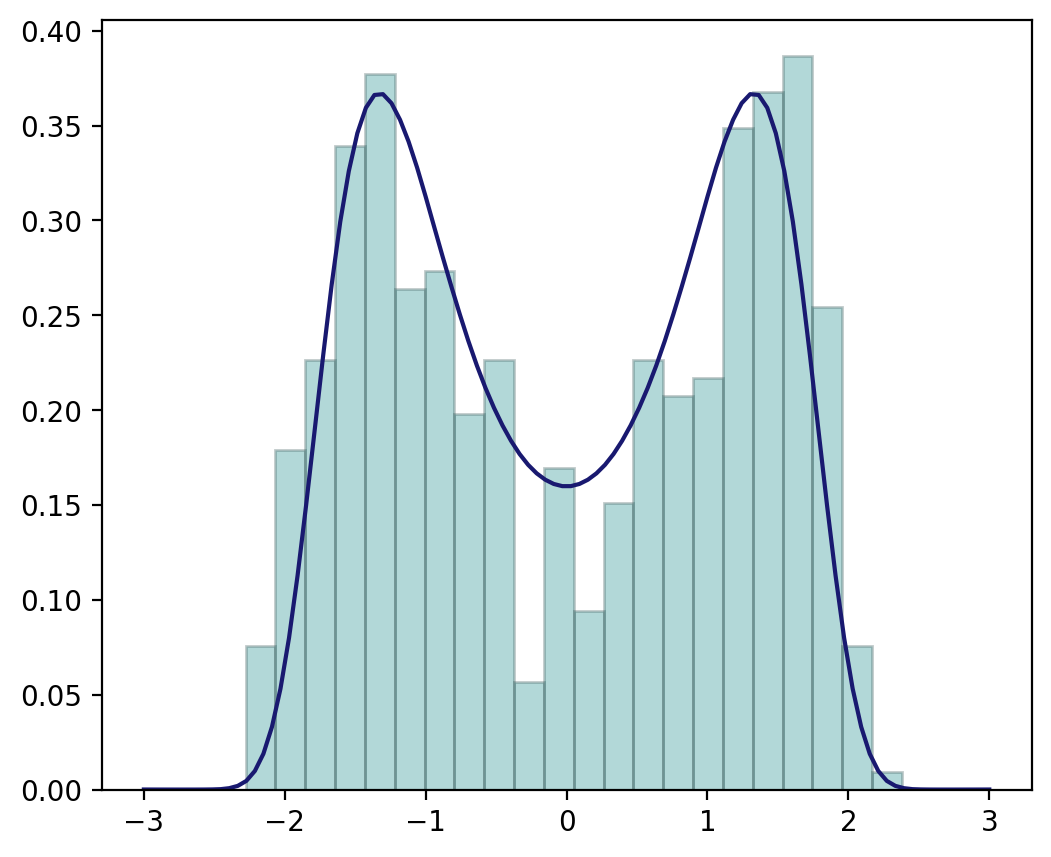}
    \includegraphics[width=5cm]{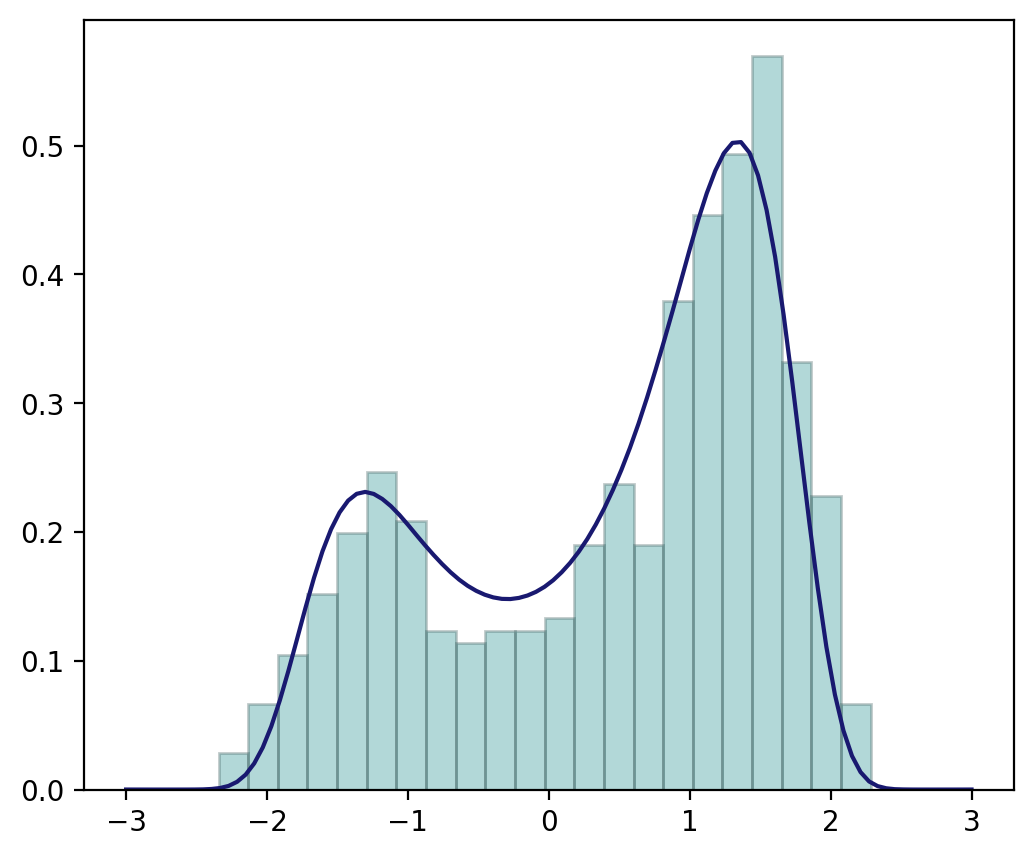}
    \includegraphics[width=5cm]{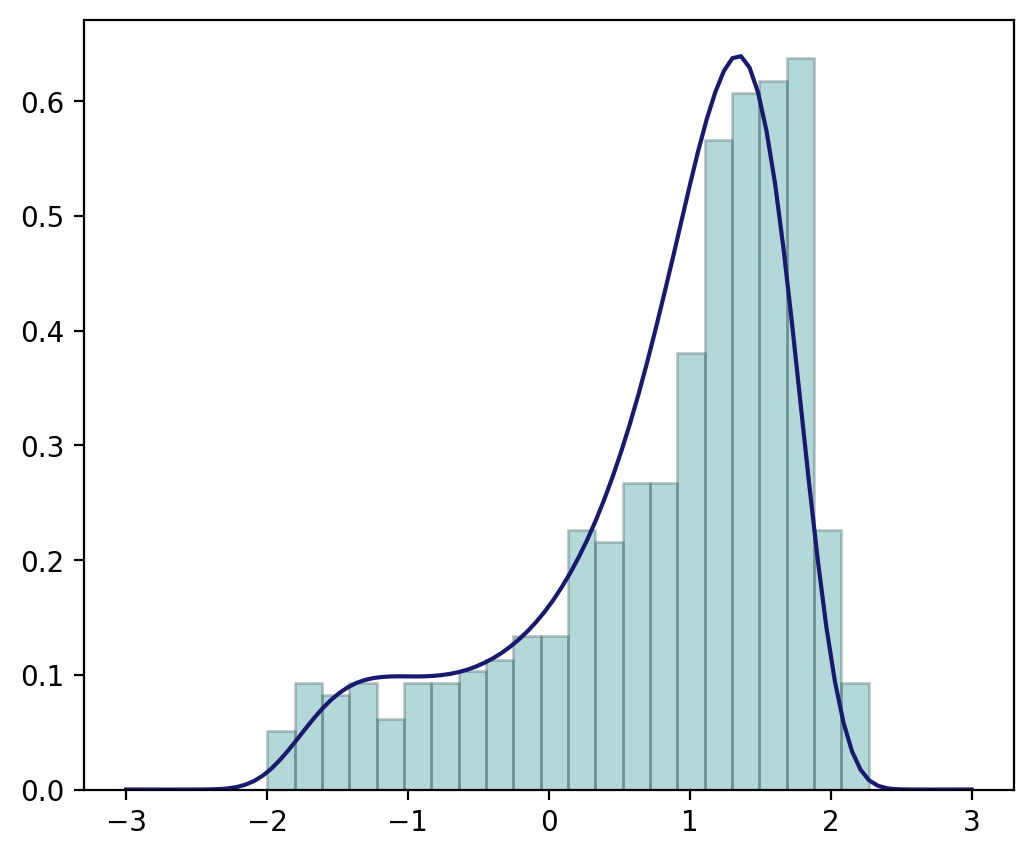}
    \caption{Histogram and method of moments approximation of  the density (blue lines) of $L_q$. First (resp.\ second, third) line: $p=0.77$ (resp.\ $p=0.87$, $p=0.92$). First (resp.\ second, third) column: $q=0.5$ (resp.\ $q=0.7$, $q=0.9$).}
    \label{fig:simulations}
\end{figure}

Each display contains a histogram of the distribution of $L_q$ as well as an approximation of the plot of the density of $L_q$. For each of the nine pictures, to obtain the histogram we have simulated $500$ paths of ERW of length $10^4$.

In order to obtain the approximated curves of the density, we make use of a classical method to approximate numerically a
measure knowing its moments, as presented in  \cite{Ga-04}. More specifically, given a finite number of moments (in our case, $300$, quickly obtained from the recurrence~\eqref{eq:recursion_moments}), we may introduce the so-called Hankel matrix, and then there is a canonical way to compute the associated Jacobi matrix. The eigenvalues and eigenvectors of this Jacobi matrix give an approximation of the density by a finite sum of Dirac measures. As a last step, we may consider a smooth (in our case, cubic) interpolation of this purely atomic measure, which yields the graphs of Figure~\ref{fig:simulations}. We kindly thank the authors of the paper \cite{ChGaKaVe-23}, in particular Slim Kammoun, for providing us with an effective code.

\subsection{A cluster decomposition of the limit variable and its relations with the computation of moments}\label{sec:randomtrees}

In this part, we first state a remarkable decomposition of the limit variable $L_q$, which is obtained using the interpretation of ERW in terms of random recursive trees and results on cluster sizes after percolation. Although this result does not clearly appear in the literature, it follows directly from arguments in \cite{Baur2015,Businger2018}.
\begin{prop}[\cite{Baur2015,Businger2018}]
\label{prop:rec-trees-dim1}
For $q\in[0,1]$ and $a=2p-1>\frac{1}{2}$, the ERW asymptotic satisfies
    \begin{equation}
    \label{eq:L-sum_details}
L_q=C_1Z_1^{(q)}+\sum_{j=2}^{\infty}C_jZ_j \quad\text{a.s.}
\end{equation}
where 
\begin{itemize}
    \item $C_1$ is a Mittag-Leffler random variable with parameter $a$ and $Z_1^{(q)}$ has Rademacher distribution with parameter $q$, independent of $C_1$; 
     \item $(Z_j)_{j\geq 2}$ are independent  Rademacher random variables with parameter $1/2$, independent of $Z_1^{(q)}$;
    \item $(C_j)_{j\geq 2}$ are  dependent random variables such that for any $j\geq 2$, $C_j\overset{\cL}{=}C_{1,j}\PAR{\beta_{\tau_j}}^a $, with
    \begin{itemize}
    \item $C_{1,j}$ is a Mittag-Leffler random variable with parameter $a$;
    \item $(\beta_k)$ denote independent Beta random variables with parameters $(1,k-1)$;
    \item $\tau_j-1$ are dependent negative binomial random variables with parameters $(j-1,1-a)$, such that $\tau_j = \tau_{j-1}+G_j$ and $G_j$ has geometric distribution of parameter $1-a$ and is independent of $\tau_{j-1}$;
    \item for each $j\geq 2$, the random variables  $C_{1,j}, \beta_{\tau_j}$ are independent.
    \end{itemize}
    \item The random variables $(Z_j)_{j\geq 1}$ are independent of the random variables $(C_j)_{j\geq 1}$.
\end{itemize}

\end{prop}

\begin{proof}
The proof follows from compiling several results on the ERW and the cluster size of random recursive trees (RRT)\ with Bernoulli bond percolation, see \cite{Businger2018} for the explicit construction and \cite{Bertoinc2021} for the definition of the ERW as a step-reinforced random walk. We give here the main arguments, but leave the details to the reader. Let $\mathcal{T}_n$ be a RRT of size $n$, and denote by $T_i$ the subtree of $\mathcal{T}_n$ rooted at $i$ after Bernoulli bond percolation (each edge is kept with probability $a$). In particular, we have a forest after percolation.  
It is known from \cite[Lem.~3.3]{Baur2015}, see also \cite[Lem.~4 and Lem.~5]{Businger2018},
that
\begin{equation*}
    \lim_{n\to\infty} n^{-a}\#\big\{j\leq n:\ j\in T_1\big\} = C_1\quad\text{a.s.},
\end{equation*}
where $C_1$ has a Mittag-Leffler distribution with parameter $a$. More generally, for $i\geq 1$,
\begin{equation*}
    \lim_{n\to\infty} n^{-a}\#\big\{j\leq n:\ j\in T_i\big\} = C_i>0 \quad\text{a.s.},
\end{equation*}
where $C_i$ has the same law as $(\beta_i)^a \cdot C_{1,i}$, where $\beta_i$ denotes a Beta variable with parameter $(1,i-1)$  and $C_{1,i}$ is a Mittag-Leffler variable with parameter $a$ independent of $\beta_i$.

Hereafter, we use the connection with the ERW to write \cite{Kursten2016,Businger2018}
\begin{equation}
\label{Sn-RRT-cluster}
S_n = \sum_{j=1}^{\infty}\vert c_n(\tau_j)\vert Z_j,
\end{equation}
where
\begin{itemize}
    \item $Z_1 \sim \cR(q)$, $(Z_j)_{j\geq2}$ is a sequence of i.i.d.\ random variables with $\cR(1/2)$-distribution;
    \item $c_n(j) = \big\{j\leq n:\ j\in T_i\big\}$ is the cluster rooted at $j$ of the RRT associated with the ERW after percolation and $|c_n(j)|$ is its size;
    \item $(\tau_j-1)_{j\geq 1}$ are dependent negative binomial random variables with respective parameters $(j-1,1-a)$, such that $\tau_j = \tau_{j-1}+G_j$, where $G_j$ has geometric distribution of parameter $1-a$ and is independent of $\tau_{j-1}$.
\end{itemize}
The $(\tau_j)_{j\geq 1}$ are here to ensure that we do not count more than once the steps of the ERW (i.e., that we are looking that the roots of the trees after percolation or at the $j$-th tree of the forest).
Gathering all of the above, we find that when $a>1/2$ both $S_n/n^a$ and $\vert c_n(\tau_j)\vert/n^a$ converge (for all $j$) and the result is proved.
\end{proof}

An application of the previous result is that the cluster decomposition \eqref{eq:L-sum_details} allows us to obtain a new proof of the recursion for the moments obtained in Theorem~\ref{thm:moments}. More precisely:

\begin{rk}
    Even though each $C_j$ in \eqref{eq:L-sum_details} can be written in distribution as a product of a Mittag-Leffler variable of parameter $a$ with an independent variable,  and even though Equation \eqref{eq:L_moments} satisfied by the moments of $L_q$ suggests a similar factorization, the asymptotic variable $L_q$ cannot be written as the product of a Mittag-Leffler variable of parameter $a$ with an independent variable having moments $(m_k)_{k\geq 0}$. The reason is that the sequence of $(m_k)_{k\geq 0}$, defined by the recursive equation~\eqref{eq:recursion_moments}, does not correspond to the sequence of moments of a random variable, since the associated Hankel matrix is not positive. The factorization of the variables $C_j$ is just an equality in law, and the choice of $C_{1,j}$ depends on $j$.
\end{rk}

\section{The multidimensional Elephant random walk}
\label{S-HD}

\subsection{Definition and first properties of the process}
\label{subsec:MERW}

Let $d\geq 1$. All vectors $\bw=(w_j)_{1\leq j\leq d}$ in this section are seen as column vectors. For vectors $\bw,\bw'\in \dR^d$, $\bw^T$ stands for the transpose vector of $\bw$, $\SCA{\bw,\bw'}$ denotes the scalar product and $\NRM{\bw}$ is the Euclidean norm.

The multidimensional ERW (MERW) $(\bS_n)_{n\geq0}$ on $\dZ^d$ is defined as follows. Let $(\be_i)_{1\leq i\leq d}$ be the canonical basis of $\dR^d$. At time zero, $\bS_0 = 0$ and at time $n = 1$, the elephant moves in one of the $2d$ directions $(\pm \be_i)_{1\leq i\leq d}$ according to a probability vector $q=(q_1,\ldots, q_{2d})$, meaning that
\begin{equation}
    \label{eq:first_value_MERW}
    \dP(\bS_1=\be_k)=q_{2k-1}\quad \text{and}\quad \dP(\bS_1=-\be_k)=q_{2k},\quad \forall k\in\BRA{1,\ldots,d}.
\end{equation} 
Afterwards, at time $n+1 \geq 2$,  the elephant chooses uniformly at random an integer $k$ among the previous times $1,\ldots,n$. 
Then, it repeats exactly its $k$-th step  with probability $p$,  or it moves  uniformly in one of the $2d-1$ remaining directions, i.e., with the same probability $\frac{1-p}{2d-1}$, where the parameter $p$ stands for the memory parameter of the MERW. The position of the elephant at time $n+1$ is given by
\begin{equation*}
    \bS_{n+1}=\bS_n+\bX_{n+1},
\end{equation*}
with $\bX_{n+1}$ being defined as the step of this random walk at time $n+1$, and satisfying
\begin{equation*}
    \bX_{n+1}=A_{n+1}\bX_{\cU(n)}.
\end{equation*}
Here $\cU(n)$ is a uniform variable on $\BRA{1,\ldots,n}$, and  $(A_{n})_{n\geq 1}$ represents a sequence of i.i.d.\ random matrices, independent of $\cU(n)$, given by
\begin{equation*}
    \dP(A_n = + I_d)=p, 
\end{equation*}
and, for $k\in\BRA{1,\ldots,d-1}$,
\begin{equation*}
  \dP(A_n = - I_d) = \dP(A_n = + \PAR{J_{d}}^k)= \dP(A_n = - \PAR{J_{d}}^k) = \frac{1-p}{2d-1},
\end{equation*}
where $I_d$ is the identity matrix of dimension $d$ and the matrix $J_{d}:=\SBRA{\be_2,\ldots,\be_d,\be_1}$. We observe that $\PAR{J_{d}}^d=I_d$.

Define in this section
\begin{equation}
\label{eq:def_a_dim-d}
    a:=\frac{2dp-1}{2d-1}.
\end{equation}
In the superdiffusive regime (characterized, as we shall see below, by $a\in\PAR{\frac{1}{2},1}$), Bercu and Laulin proved the convergence of the  normalized MERW.
\begin{thm}[Theorem~3.7 in \cite{BercuLaulin2019}]\label{thm:limitERW}
Let $(\bS_n)_{n\geq 0}$ be the $d$-dimensional ERW with memory parameter $p\in[0,1]$ and initial probability vector $q\in[0,1]^{2d}$ as in \eqref{eq:first_value_MERW}. 
When $a>1/2$, there exists a non-degenerate random vector $\bL$ in $\dR^d$ such that
\begin{equation}\label{eq:limitL-2d}
\lim_{n\to \infty}\frac{\bS_n}{n^a}=\bL\quad \text{ a.s.}
\end{equation}
\end{thm}

We now give some properties of the random vector $\bL$.
From \cite[Rem.~3.5]{BercuLaulin2019} giving the first moments of the MERW, we deduce the following lemma. Notice that we recover the result of \cite[Thm~3.8]{BercuLaulin2019} in the case $q$ is the uniform distribution on the $2d$ directions.
\begin{lem}
\label{lem:cov_matrix}
Under the same hypotheses as in Theorem~\ref{thm:limitERW}, the limit $\bL$ in \eqref{eq:limitL-2d} satisfies\begin{align*}
\dE\SBRA{\bL}&=\frac{1}{\Gamma(a+1)}
\begin{pmatrix}
    q_1-q_{2}\\
q_3-q_{4}\\
\vdots\\
q_{2d-1}-q_{2d}\\
\end{pmatrix},\medskip\\
 \dE\SBRA{\bL\bL^T}&=\frac{1}{\Gamma(2a+1)}D +\frac{1-a}{da(2a-1)\Gamma(2a)}I_d,
\end{align*}
where $D$ is the diagonal matrix $D=\mathrm{diag}\PAR{q_1+q_2,q_3+q_{4},\ldots, q_{2d-1}+q_{2d}}$.
\end{lem}

\begin{proof}
When the first step probability vector $q$ is $(1,0,\ldots,0)$, Remark~3.5 in \cite{BercuLaulin2019} states that
\begin{align*}
\dE\SBRA{\bL}&=\frac{1}{\Gamma(a+1)}\be_1,
\\ \dE\SBRA{\bL\bL^T}&=\frac{1}{\Gamma(2a+1)}\PAR{\be_1\be_1^T -\frac{1}{d}I_d}+\frac{1}{d(2a-1)\Gamma(2a)}I_d\\
&=\frac{1}{\Gamma(2a+1)}\be_1\be_1^T +\frac{1-a}{da(2a-1)\Gamma(2a)}I_d.
\end{align*}
We have similar results  when $q$ satisfies $q_k=1$ and $q_i=0$ for $i\neq k$, for any $k\in\BRA{1,\ldots,2d}$. Let us now consider a general distribution $q=(q_1,\ldots, q_{2d})$ for the first step. The limit $\bL$ starting from the probability vector $q$ can be written as the mixture of random vectors 
\begin{equation}
\label{eq:mixture}
   \bL=\sum_{k=1}^{2d}\bL_{(k)}\ind_{\overline{q}_{k-1}\leq U<\overline{q}_k},
\end{equation}
where $\bL_{(k)}$ the limit of the MERW starting from a first step in the direction $(-1)^{k-1}\be_{\lfloor \frac{k-1}{2}\rfloor +1}$, $U$ an independent uniform variable on $[0,1]$, $\overline{q}_0=0$ and $\overline{q}_k=\sum_{i=1}^kq_i$. Since
\begin{align*}
\dE\SBRA{\bL_{(k)}}&=\frac{(-1)^{k-1}}{\Gamma(a+1)}\be_{\lfloor \frac{k-1}{2}\rfloor +1} 
\\ \dE\SBRA{\bL_{(k)}\bL_{(k)}^T}
&=\frac{1}{\Gamma(2a+1)}\be_{\lfloor \frac{k-1}{2}\rfloor +1}\be_{\lfloor \frac{k-1}{2}\rfloor +1}^T +\frac{1-a}{da(2a-1)\Gamma(2a)}I_d,
\end{align*}
the conclusion follows.
\end{proof}

Let $\bZ_1^{(q)}$ be a random variable with distribution $q$ on the $2d$ directions $\PAR{\pm\be_i}_{1\leq i \leq  d}$ and  $(\bZ_j)_{j\geq 2}$ be a sequence of i.i.d.\ random variables with uniform distribution on the directions, independent of $Z_1^{(q)}$. As in dimension $1$ (see Proposition~\ref{prop:rec-trees-dim1}), 
    the superdiffusive limit $\bL$ of MERW can be written
\begin{equation}\label{eq:rec-trees-dimd}
  \bL  =  C_1\bZ_1^{(q)}+\sum_{j=2}^{\infty}C_j\bZ_j,
\end{equation}
where the variables $(C_j)_{j\geq 1}$ are defined in Proposition~\ref{prop:rec-trees-dim1} and are independent of $\bZ_1^{(q)}$ and $(\bZ_j)_{j\geq 2}$.

The following lemma shows that the coordinates of the vector $\bL$ are not independent, even though its covariance matrix is diagonal by Lemma~\ref{lem:cov_matrix}. As an important consequence for us, in order to prove that the vector $\bL$ has a density it is not sufficient to show that each coordinate has a density.

\begin{lem}
\label{lem:dependence}
We assume that $q$ is the uniform distribution on the $2d$ directions and $a>\frac{1}{2}$. Then the coordinates of $\bL=\PAR{L_{i}}_{1\leq i\leq d}$ have the same distribution and
\begin{equation*}
 \sum_{i=1}^dL_{i}=\sum_{j=1}^{\infty}C_j \zeta_j\overset{\cL}{=}L_{1/2},
\end{equation*}
where 
\begin{itemize}
    \item $(C_j)_{j\geq 1}$ are defined as in Proposition~\ref{prop:rec-trees-dim1};
    \item $(\zeta_j)_{j\geq 1}$ are i.i.d.\ Rademacher $\cR(1/2)$ random variables, independent of the above variables;
    \item $L_{1/2}$ is the asymptotic variable in dimension $1$ of ERW with first step probability $1/2$.
\end{itemize}
\end{lem}

\begin{proof}
    When $q$ is the uniform distribution, for each $j\geq 1$, the random vectors $\bZ_j$ in \eqref{eq:rec-trees-dimd} can be written $\bZ_j = \zeta_j \bM_j$, where $\zeta_j$ is a Rademacher $\cR(1/2)$ random variable and $\bM_j$ is a multinomial $\cM(1,(\frac{1}{2d},\ldots,\frac{1}{2d}))$ random vector of $\dR^d$, with $\zeta_j$ and $\bM_j$ independent.

We easily deduce that each coordinates of $\bL$ have the same distribution. Moreover,
the sum of the coordinates of each $\bM_j$ being equal to $1$, the lemma is proved.
\end{proof}

We deduce from this lemma that  the coordinates of $\bL$ are not independent. Indeed, when $d=2$, for each $j$ we have
$\bZ_j = \zeta_j(B_j, 1-B_j)^T$, where $\zeta_j$ is a Rademacher $\cR(1/2)$ random variable, $B_j$ a Bernoulli $\cB(1/2)$ random variable, and both are independent. Consequently, for $d=2$ and $q=(1/4,1/4,1/4,1/4)$,
\begin{equation*}
    \bL  =\sum_{j=1}^{\infty}C_j \zeta_j
    \begin{pmatrix} B_j\\1-B_j\end{pmatrix}
    \overset{\cL}{=}\begin{pmatrix}R\\L_{1/2}-R\end{pmatrix},
\end{equation*}
where $L_{1/2}$ is the limiting random variable in dimension $1$ and $R=\sum_{j=1}^{\infty}C_j \zeta_j B_j$ is a priori not independent of $L_{1/2}$.

\subsection{Urn process associated to the MERW}
\label{subsec:urn_process_MERW}

We now work with the space $\dR^{2d}$ endowed with the canonical basis $(\be_i)_{1\leq i\leq 2d}$.
Similarly as in the dimension $1$ case, we associate to the MERW a P\'olya-type urn process $\bU=\PAR{\bU(n)}_{n\geq 0}$ with $2d$ different colors (representing the $2d$ possible directions of the MERW), with the following replacement matrix:
\begin{equation}\label{eq:replacement-matrix-A}
A=\begin{cases}
I_{2d}&\text{ with probability }p,\\
\PAR{J_{2d}}^{k}&\text{ with probability }\frac{1-p}{2d-1}\text{ for each } k\in\BRA{1,\ldots, 2d-1}, 
\end{cases}
\end{equation}
where $I_{2d}$ is the identity matrix in dimension $2d$ and the matrix $
J_{2d}:=\SBRA{e_2,e_3,\ldots,e_{2d},e_1}$, see \cite{Bertenghi2020}.
We remark that $\PAR{J_{2d}}^{2d}=I_{2d}$ and 
 \begin{equation*}
  \dE[A]=pI+\frac{1-p}{2d-1}\begin{pmatrix}
      0&1&1&\ldots&1\\
      1&0&1&\ldots&1\\
      1&1&\ddots&\ddots&\vdots\\
      \vdots&\vdots&\ddots&\ddots&1\\
      1&1&\ldots&1&0
  \end{pmatrix}.
 \end{equation*}
 The matrix  $\dE[A]$ is diagonalizable and admits the eigenvalues
 \[
 \lambda_1 = 1,\quad \lambda_2 = \lambda_3 =\ldots= \lambda_{2d} = \frac{2dp-1}{2d-1},
 \]
with respective unit eigenvectors $\bv_1=\frac{1}{2d}\sum_{i=1}^{2d}\be_i$, $\bv_2=\frac{1}{2}(\be_1-\be_2)$, $\bv_3=\frac{1}{2}(\be_1-\be_3)$, up to $\bv_{2d}=\frac{1}{2}(\be_1-\be_{2d})$.
The following result may be found in Theorems~3.9, 3.24 and 3.26 of \cite{Janson2004}.
 \begin{thm}[\cite{Janson2004}]\label{thm:Janson-dimd}
 Let $a=\frac{2dp-1}{2d-1}$. We assume
  $a>1/2$. Assume further that $\bU(0)=\be_k$ for some $k\in\BRA{1,\ldots, 2d}$,  i.e., the urn process starts from a unique ball of color $k$. Then we have 
 \begin{equation}\label{eq:conv-Udimd}
     \lim_{n\to\infty} \frac{\bU(n) - n\bv_1}{n^a}= \sum_{i=2}^{2d} W_i(\be_k)\bv_i  \quad\text{a.s.},
 \end{equation}
where $W_i(\be_k)$ are real-valued random variables depending on $k$.

Moreover, the expectation of $\bY_{(k)}:=\sum_{i=2}^{2d}W_i(\be_k)\bv_i$ is given by
\begin{equation}
\label{eq:expectation-W-2d}
    \dE[\bY_{(k)}] = \frac{1}{\Gamma(a+1)}\PAR{\be_k - \bv_1}.
\end{equation}
 \end{thm}

We denote by $\bU_{(k)}$ the urn  process starting from one ball of the $k$-th color. We adapt the idea used for an urn with two colors and divide the tree structure at time $1$ into two subtrees: the first starting from a node of color $k$, and the second starting from a node with color $k$  with probability $p$, or from another color (randomly chosen)\ with probability $1-p$. For $i\in\BRA{1,2}$, $n\geq1$,  $D_i(n)$ denotes the number of leaves  at time $n$ of the $i$-th subtree. We recall that these numbers represent the time inside each subtree. Following the ideas of Section~\ref{sec:tree-struct-dim1}, we deduce the lemma below.
 \begin{lem}\label{lem:zeta}
 For any $ n\geq 1$,
 \begin{equation}
 \label{eq:urn-eq-2d}
     \bU_{(k)}(n)\overset{\cL}= \bU^{(1)}_{(k)}\PAR{D_1(n)-1} + \sum_{j=1}^{2d} \zeta_{j,(k)}\bU^{(2)}_{(j)}\PAR{D_2(n)-1},
 \end{equation}
where $\zeta_{(k)}=\PAR{\zeta_{j,(k)}}_{1\leq j\leq 2d}$ denotes a random vector such that $\zeta_{(k)}=\be_k$ with probability $p$ and $\zeta_{(k)}=\be_{k'}$ with probability $\frac{1-p}{2d}$, for each $k'\neq k$. 
\end{lem}
Using the above almost sure convergence \eqref{eq:conv-D} and Theorem~\ref{thm:Janson-dimd}, this leads to the system of equations
 \begin{equation}
     \label{eq:bX-2d}
     \bY_{(k)}\overset{\cL}= V^a\bY^{(1)}_{(k)} + (1-V)^a\sum_{j=1}^{2d} \zeta_{j,(k)}\bY^{(2)}_{(j)},\quad k\in\BRA{1,\ldots,2d},
 \end{equation}
where $\bY_{(k)}$ is the limit vector of the urn process started with one ball of the $k$-th color. In Equation~\eqref{eq:bX-2d}, $V$ is a uniform random variable on the segment $[0,1]$, $\bY^{(1)}_{(k)}$ and $\bY^{(2)}_{(k)}$ are copies of $\bY_{(k)}$, and all random variables are independent.

\subsection{Link between the asymptotics of the urn process and the MERW}
\label{subsec:link_urn_MERW}

We consider the urn process $\bU$ defined as above, starting from $\bU(0)$ chosen according to a probability vector $q$.
We can easily check that the $d$-dimensional process $(\bS_n)_{n\geq 0}$ defined for all $n\geq 1$ by
  \begin{equation*}
  \left\{\begin{array}{lclcl}
     S_{1,n} &=& U_1(n)-U_{2}(n) &=& 2\bv_2^T\bU_n  ,\\
     S_{2,n} &=& U_3(n)-U_4(n) &=& 2(\bv_4-\bv_3)^T\bU_n, \\ 
     &\vdots &&\vdots& \\
     S_{d,n} &=& U_{2d-1}(n)-U_{2d}(n) &=& 2(\bv_{2d}-\bv_{2d-1})^T\bU_n,
     \end{array}\right.
 \end{equation*}
is a MERW with memory parameter $p$ and initial distribution $q$.

We denote by $\bL_{(k)}$ the limit random vector of the MERW when the first step is in the $k$-th direction. When the initial distribution is given by a general probability vector $q$, the limit $\bL$ of the MERW can be written as a discrete mixture of $\PAR{\bL_{(1)},\ldots,\bL_{(2d)}}$, see \eqref{eq:mixture}.

 The almost sure convergences \eqref{eq:limitL-2d} and \eqref{eq:conv-Udimd} imply
   \begin{align}
   \nonumber
  \bL_{(k)}&= \lim_{n\to\infty} \frac{1}{n^a} \bS_n = \lim_{n\to\infty} \frac{1}{n^a} 
    \begin{pmatrix}  U_1(n)-U_2(n)
    \\
    U_3(n)-U_4(n)
    \\\vdots\\
    U_{2d-1}(n)-U_{2d}(n)
    \end{pmatrix} 
 =\begin{pmatrix}  Y_{1,(k)}-Y_{2,(k)}\\ Y_{3,(k)}-Y_{4,(k)}\\\vdots\\
Y_{2d-1,(k)}-Y_{2d,(k)}
\end{pmatrix}
\\[0.3cm]
&= 
    \begin{pmatrix}  W_2(\be_k) +\frac{1}{2}\sum_{k=3}^{2d}W_k(\be_k)
    \\[0.2cm]\frac{W_4(\be_k)-W_3(\be_k)}{2}\\
   \vdots\\\frac{W_{2d}(\be_k)-W_{2d-1}(\be_k)}{2}
   \end{pmatrix}
=\begin{pmatrix}  
1&\frac{1}{2}&\frac{1}{2}&\frac{1}{2}&\cdots &\cdots&\frac{1}{2}
\\[0.2cm]
0&-\frac{1}{2}&\frac{1}{2}&0&0&\ldots&0&
\\
0&0&0&-\frac{1}{2}&\frac{1}{2}&\ddots&\vdots
\\
\vdots&&&\ddots&\ddots&\ddots&0
\\
0&\cdots&\cdots&\cdots&0&-\frac{1}{2}&\frac{1}{2}
\end{pmatrix}
\bW(\be_k)  \quad\text{a.s.} 
\label{eq:L-W-dimd}
\end{align}
We notice that this last matrix of size $d\times (2d-1)$ is of rank $d$. Thus, if we prove that $\bW:=(W_2,\ldots,W_{2d})$ has a smooth positive density function on $\dR^{2d-1}$, as what we have done in dimension $1$, we will immediately deduce that the distribution of $\bL$ admits a smooth positive density function on $\dR^d$.  Consequently, we will focus our study on  the random vectors $\bW$. 

\subsection[Uniqueness and fixed-point equations in dimension d]{Uniqueness and fixed-point equations in high dimension}\label{sec:fixed-point-dimd}

We first study the system \eqref{eq:bX-2d} of fixed-point equations satisfied by the limit process, and in Proposition~\ref{prop:existence_unique_fpes} we prove existence and uniqueness of the associated solution in the space of random vectors with a finite second-order moment and a given expectation. Then in Proposition~\ref{prop:fixed-point-dimd} we derive a second fixed-point equation, which will be the key-point for studying the distribution of the limit process.

\subsubsection*{Existence and uniqueness of the solution to the system \eqref{eq:bX-2d}}

Equation \eqref{eq:bX-2d} ensures that, for all $j,k\in\BRA{1,\ldots,2d}$
\begin{equation}
\label{eq:fixed-point-X-color}
    Y_{j,(k)} = V^a Y_{j,(k)}^{(1)}+(1-V)^a \sum_{i=1}^{2d-1}A_{i,j}Y_{j,(i)}^{(2)},
\end{equation}
where $A=(A_{i,j})$ has the same law as the random replacement matrix \eqref{eq:replacement-matrix-A} of the urn process and $Y_{j,(k)}$ is the $j$-th coordinate of the limiting vector $\bY_{(k)}$ starting with one ball of the $k$-th color. Hence, for $\bY^j := \PAR{Y_{j,(k)}}_{1\leq k\leq 2d}$,
\begin{equation}
\label{eq:fixed-point-X-color-vector}
    \bY^{j} = V^a \bY^{j,(1)}+(1-V)^a A \bY^{j,(2)},
\end{equation}
with $\bY^{j,(1)}$ and $\bY^{j,(2)}$ independent copies of $ \bY^{j}$, and the above holds jointly for all $j\in\BRA{1,\ldots,2d}$ with the same $V$ and $A$.
\begin{prop}
\label{prop:existence_unique_fpes}
Let $a\in\PAR{1/2,1}$.
    Given $\bm=(m_2,\ldots,m_{2d})\in \dR^{2d-1}$, we define $\cP_2(\bm)$ as the space of probability measures on $\dR^{2d}$ with finite  second-order moment and such that for all $\mu\in \cP_2(\bm)$,
    \begin{equation*}
    \int_{\dR^{2d}} x\mu(\dd x)=\sum_{i=2}^{2d} m_i \bv_i.
    \end{equation*}
Then there exists a unique solution in $\cP_2(\bm)$ to the fixed-point equation \eqref{eq:fixed-point-X-color-vector}, for all $j\in\BRA{1,\ldots,2d}$.
\end{prop}
Thanks to the convergence Theorem~\ref{thm:Janson-dimd}, we know that there is a (not necessarily unique)\ solution to \eqref{eq:bX-2d} when $a>1/2$, and its expectation is given by \eqref{eq:expectation-W-2d}.
\begin{proof} 
As in the proof of Theorem~\ref{thm:dim1_solution_dist_eq}, we consider the Wasserstein distance $d_\cW$ on $\cP_2(m)$.
Then we define
\begin{align*}
    H :\  \cP_2(\bm) &\longrightarrow \cP_2(\bm) \\
     \mu &\longmapsto \cL\PAR{V^a\bY^{(1)}+(1-V)^a A\bY^{(2)}},
\end{align*}
where $\bY^{(1)},\bY^{(2)}$ have law $\mu$, $V$ is a uniform random variable on $[0,1]$, the random matrix $A$ is the same matrix \eqref{eq:replacement-matrix-A} as in the urn process, and all the variables are independent from each other. As $\dE[V^{a}]=\frac{1}{1+a}$, and for all $i\in\{2,\ldots,2d\}$, $\dE[A]\bv_i=a\bv_i$, we easily deduce that for all $\mu\in\cP_2(\bm)$, $H(\mu)\in\cP_2(\bm)$.

We now show that $H$ is Lipschitz-continuous for the Wasserstein metric.
Let 
$\mu,\nu\in\cP_2(\bm)$, and $\bigl(\bY^{(1)},\widetilde{\bY}^{(1)}\bigr)$ and $\bigl(\bY^{(2)},\widetilde{\bY}^{(2)}\bigr)$ two independent couples with marginal distributions $(\mu,\nu)$.
We then compute
\begin{align*}
&d_\cW(H\mu,H\nu)^2 \\&\quad \leq \dE\SBRA{\NRM{V^a(\bY^{(1)}-\widetilde{\bY}^{(1)})+(1-V)^aA(\bY^{(2)}-\widetilde{\bY}^{(2)})}^2} \\
&\quad \leq \dE[V^{2a}]\dE\SBRA{\NRM{\bY^{(1)}-\widetilde{\bY}^{(1)}}^2}+\dE[(1-V)^{2a}]\dE\SBRA{\NRM{A\PAR{\bY^{(2)}-\widetilde{\bY}^{(2)}}}^2} \\
&\quad \leq \frac{2}{1+2a} \dE\SBRA{\NRM{\bY^{(1)}-\widetilde{\bY}^{(1)}}^2},
\end{align*}
because $\NRM{A\bY}=\NRM{\bY}$ for all $\bY\in\dR^{2d}$.
Consequently, taking the infimum over all couples $\bigl(\bY^{(1)},\widetilde{\bY}^{(1)}\bigr)$ with marginal distributions $(\mu,\nu)$, we deduce 
\begin{equation*}
   d_\cW(H\mu,H\nu)  \leq  \sqrt{\frac{2}{1+2a}} d_\cW(\mu,\nu)
\end{equation*}
and $H$ is a contraction. The conclusion follows.
\end{proof}
Equation~\eqref{eq:fixed-point-X-color-vector} is not a fully convenient (fixed-point)\ equation to study the distributions of $\bW$ and $\bL$, as we cannot easily relate $\bY^j$ to the $\bL_{(k)}$'s. Besides, for any $k\geq 1$, by definition of $\bY_{(k)}$ in Theorem~\ref{thm:Janson-dimd} and by definition of the eigenvectors $(\bv_i)_{1\leq i\leq 2d}$, we easily observe that
\begin{equation}\label{eq:lienX-W}
2\bY_{(k)}=\PAR{\sum_{i=2}^{2d}W_i(\be_k)} \, \be_1-\sum_{i=2}^{2d}W_i(\be_k)\, \be_i.
\end{equation}
 The support of $\bY_{(k)}$ is thus included in the hyperplane $\bigl\{x\in\dR^{2d}:\sum_{i=1}^{2d}x_i=0\bigr\}$, and $\bY_{(k)}$ cannot have a density with respect to the Lebesgue measure on $\dR^{2d}$.
However, the fixed-point equation~\eqref{eq:fixed-point-X-color-vector} will be useful to study the moments of $\bY$ and $\bL$ in Section~\ref{sec:moments-dimd}.
 
In the next paragraph, we introduce another fixed-point equation, which will be the key-point to show that $\bW$ and $\bL$ have a positive density with respect to the Lebesgue measure.

\subsubsection*{A second fixed-point equation in high dimension}
We now focus on the study on $\bU_{(1)}=\PAR{\bU_{(1)}(n)}_{n\geq 0}$, i.e., the urn process starting from a unique ball of color $1$: $\bU_{(1)}(0)=\be_1$.  The other cases are similar by symmetry of the model.
Indeed, we notice that the urn process starting from a unique ball of any color $k\in\BRA{1,\ldots,2d }$ is equal to
\begin{equation*}
    \PAR{\PAR{J_{2d}}^{k-1}\bU_{(1)}(n)}_{n\geq 0}.
\end{equation*}
More generally, for $i,k\in\BRA{1,\ldots, 2d}$, we have $\bU_{(i)}(n)=J_{2d}^{i-k}\bU_{(k)}(n)$. 
For clarity of the proofs, we will omit the subscript in the sequel, and we will denote $\bU:=\bU_{(1)}$, $\bY:=\bY_{(1)}$, and for $i\in\BRA{2,\ldots,2d}$, $W_i:=W_i(\be_1)$.

 We notice that Equation~\eqref{eq:bX-2d} leads to
 \begin{equation}
     \label{eq:bXdimd}
     \bY\overset{\cL}= V^a\bY^{(1)} + (1-V)^a\sum_{j=1}^{2d} \zeta_{j}J^{j-1}\bY^{(2)},
 \end{equation}
 where $\bY$ is the limit vector of the urn process started with one ball of color $1$ (see Theorem~\ref{thm:Janson-dimd}), and $\zeta:=\zeta_{(1)}$ is defined in Lemma~\ref{lem:zeta}.

We now give a useful fixed-point equation in dimension $2d-1$. 
Define $K$ as the following variant of a companion matrix of dimension $(2d-1)\times(2d-1)$:
\begin{equation}
\label{eq:def_matrix_K}
     K=\begin{pmatrix}
         -1&-1&\cdots &\cdots&-1\\
         1&0&\cdots&\cdots &0\\
         0&\ddots&\ddots&&\vdots\\
         \vdots&\ddots&\ddots&\ddots&\vdots\\
         0&\cdots&0&1&0
     \end{pmatrix}.
     \end{equation}
Observe that $K$ is a cyclic matrix: $K^{2d}=I_{2d-1}$.
\begin{prop}
\label{prop:fixed-point-dimd}
    The limit vector $\bW=(W_2,\ldots,W_{2d})\in \dR^{2d-1}$, defined in Theorem~\ref{thm:Janson-dimd} for $\bU(0)=\be_1$, satisfies the fixed-point equation
     \begin{equation}
     \label{eq:fixed-point-W-dimd}
\bW \overset{\cL}{=}V^{a}
\bW^{(1)}
+(1-V)^{a} B
\bW^{(2)},
 \end{equation}
where
\begin{itemize}
    \item $V$ denotes a uniform random variable on $[0,1]$;
    \item $\bW^{(1)}$ and $\bW^{(2)}$ are copies of  $\bW$;
  \item $B$ is a random matrix equal to $B=K^{k}$ with probability $p_k$ for each $ k\in\BRA{0,\ldots, 2d-1}$, 
 where $p_0=p$ and $p_1=\ldots=p_{2d-1}=\frac{1-p}{2d-1}$;
\item all the variables $V$, $\bW^{(1)}$, $\bW^{(2)}$ and $B$ are assumed independent.
\end{itemize}
\end{prop}

\begin{proof}
By \eqref{eq:lienX-W}, we note that  $\bW=(W_2,\ldots,W_{2d})=-2\PAR{Y_2,\ldots,Y_{2d}}$. Then, we deduce from \eqref{eq:bXdimd} that $\bW$ satisfies the following fixed-point equation
\begin{equation*}
\bW
\overset{\cL}{=}V^{a}
\bW^{(1)}+(1-V)^{a}\widetilde \bW^{(2)},
\end{equation*}
where $\widetilde \bW^{(2)}$ is the random vector
\begin{align*}
\widetilde \bW^{(2)}&=
\left(\begin{array}{lclclcl}
(\zeta_1-\zeta_2)W_2^{(2)}&+&(\zeta_3-\zeta_2)W_
{2d}^{(2)}
&+&\cdots &+&(\zeta_{2d}-\zeta_2)W_3
 \\
 (\zeta_1-\zeta_3)W_3^{(2)}&+&(\zeta_2-\zeta_3)W_2^{(2)}
 &+&\cdots&+&(\zeta_{2d}-\zeta_3)W_4^{(2)}
 \\
 \qquad \vdots & &  \qquad \vdots& & & &\qquad \vdots
\\
 (\zeta_1-\zeta_{2d})W_{2d}^{(2)}&+&(\zeta_2-\zeta_{2d})W_{2d-1}^{(2)}
 &+&\cdots&+&(\zeta_{2d-1}-\zeta_{2d})W_2^{(2)}
 \end{array}\right)
 \\[0.5cm]
 &=
\begin{pmatrix}
\zeta_1-\zeta_2&\zeta_{2d}-\zeta_2&\ldots&\zeta_4-\zeta_2&\zeta_3-\zeta_2
 \\
 \zeta_2-\zeta_3&\zeta_1-\zeta_3&\ldots&\zeta_{5}-\zeta_3&\zeta_4-\zeta_3
 \\
 &\vdots&&\vdots&
\\
\zeta_{2d-1}-\zeta_{2d}&\zeta_{2d-2}-\zeta_{2d}&\ldots&
\zeta_2-\zeta_{2d}& \zeta_1-\zeta_{2d}
 \end{pmatrix}\bW^{(2)}.
 \end{align*}
The conclusion follows by computing the distribution of latest matrix from the distribution of $\zeta=\zeta_{(1)}$, defined in Lemma \ref{lem:zeta}.
\end{proof}

\subsection[Existence of a density]{Existence of a density for $\bW$}
\label{subsec:density_d}

In this section, we assume that
\begin{equation}
\label{eq:support_W_hypothesis}
    \Supp(\bW)=\dR^{2d-1}.
\end{equation}
The support will be carefully studied in Section~\ref{sec:support}.
We follow the same scheme of proof as in dimension $1$. Consequently, our aim is to study the integrability of the characteristic function of $\bW$, defined on $\dR^{2d-1}$ by
\begin{equation*}
   \varphi_\bW(t):=\dE\SBRA{\ee^{i\SCA{t,\bW}}}.
\end{equation*}

\begin{thm}\label{thm:density-W-dimd}
     Assuming \eqref{eq:support_W_hypothesis}, let $\bW$ be a non-zero solution of \eqref{eq:fixed-point-W-dimd}  with a general  probability vector $(p_k)_{0\leq k\leq 2d-1}$ satisfying $p_k>0$, for all $ k\in\BRA{0,\ldots, 2d-1}$. Then 
     \begin{itemize}
         \item for any $\rho\in(0,1/a)$, and any $k\geq 1$, there exists $C>0$ such that for any $t\in\dR^{2d-1}$, 
         \begin{equation}
        \label{eq:upper_bound_cf}
             \ABS{\varphi_\bW(t)} \leq \frac{C}{\NRM{t}^{k\rho}};
         \end{equation}
         \item the distribution of $\bW$ has a bounded smooth density function with respect to the Lebesgue measure on $\dR^{2d-1}$.
     \end{itemize}
 \end{thm}

To prove Theorem~\ref{thm:density-W-dimd}, we adapt the arguments of the proof of Theorem~\ref{thm:fct-carac-X} to a higher dimension. The proof is divided in three different lemmas. 
 
\begin{lem}\label{lem:phiW}
Let $\bW$ be a solution of the fixed-point equation \eqref{eq:fixed-point-W-dimd} satisfying the assumptions of Theorem~\ref{thm:density-W-dimd}. Its characteristic function satisfies $\ABS{\varphi_\bW(t)}<1$ for all non-zero $t\in\dR^{2d-1}$.
\end{lem}

\begin{proof}
Using our assumption \eqref{eq:support_W_hypothesis} on the support of $\bW$,  we remark that for $ t_0\neq 0$, $\mathrm{Supp}(\SCA{t_0,\bW})=\dR$. Consequently, the proof is exactly the same as in dimension $1$ (see Step~1 in the proof of Theorem~\ref{thm:fct-carac-X}).
\end{proof}

\begin{lem}
\label{lem:phiW-to-0}
Let $\bW$ be a solution of the fixed-point equation \eqref{eq:fixed-point-W-dimd} satisfying the assumptions of Theorem~\ref{thm:density-W-dimd}. Then ${\lim_{\NRM{t}\to \infty}\varphi_\bW(t)=0}$.
\end{lem}

\begin{proof}
    Let $t\in\dR^{2d-1}$. By \eqref{eq:fixed-point-W-dimd} and by conditioning on $V$, we remark that
    \begin{equation}\label{eq:fixed-point-phiW}
\varphi_\bW(t)=p_0\dE\SBRA{\varphi_\bW(V^{a}t)\varphi_\bW((1-V)^{a}t)} +\sum_{k=1}^{2d-1}p_k\dE\SBRA{\varphi_{\bW}(V^{a}t)\varphi_{K^k\bW}\PAR{(1-V)^{a}t}}.
    \end{equation}
    Consequently,
    \[
    \ABS{\varphi_\bW(t)}\leq p_0\dE\SBRA{\ABS{\varphi_\bW(tV^{a})}\ABS{\varphi_\bW(t(1-V)^{a})}} +(1-p_0)\dE\SBRA{\ABS{\varphi_{\bW}(tV^{a})}}.
    \]
    We conclude as in dimension $1$ (Step~$2$ in the proof of Theorem~\ref{thm:fct-carac-X}), that 
    \begin{equation}
    \label{eq:reasoning_pa}
    \limsup_{\NRM{t}\to \infty}\ABS{\varphi_\bW(t)}\in\BRA{0,1}.
    \end{equation}
        We now prove the result by contradiction, assuming that the above limit \eqref{eq:reasoning_pa} is $1$. The difficulty is that $K$ is not an isometric matrix ($\NRM{K\bw}\neq \NRM{\bw}$), however it satisfies $K^{2d}=I_{2d-1}$.

Adapting Step~$3$ in the proof of Theorem~\ref{thm:fct-carac-X} in dimension $1$, and the idea used in the proof of \cite[Thm~7.1]{Chauvin-Liu-Pouyanne2014}, we introduce the function $\psi_\bW$  defined for $r\geq0$ by
    \begin{equation*}   \psi_\bW(r):=\max_{\NRM{t}=r}\max_{k\in\BRA{0,\ldots,2d-1}}\ABS{\varphi_{K^k\bW}(t)}
    =\max_{\NRM{t}=r}\max_{k\in\BRA{0,\ldots,2d-1}}\ABS{\varphi_{\bW}\PAR{\PAR{K^k}^T t}}.
    \end{equation*}
   By continuity of the function $\ABS{\varphi_\bW}$, compactness of the subset $\{t\in\dR^{2d-1}:\NRM{t}=r\}$, invertibility of $K$ and Lemma~\ref{lem:phiW}, we deduce that $\psi_\bW(0)=1$  and $\psi_\bW(r)<1$ for all $r>0$.
    
    By assumption of the proof by contradiction, ${\limsup_{r\to \infty}\psi_\bW(r)=1}$. 
      The function $\psi_\bW$ is continuous with $\psi_\bW(0)=1$, and for all $r>0$,  $\psi_\bW(r)<1$. Let $r_0>0$ be fixed and $\varepsilon>0$ such that $\psi_\bW(r_0)< 1-\varepsilon$. By the intermediate value theorem, there exists $(r_1,r_2)=(r_1(\varepsilon),r_2(\varepsilon))$  such that $0<r_1<r_0<r_2$ and
    \[
    \psi_\bW(r_1)=\psi_\bW(r_2)=1-\varepsilon \quad \text{ and }\quad \psi_\bW(r)\leq 1-\varepsilon,\quad \forall r\in(r_1,r_2).
    \]
    Moreover, taking a sequence $(\varepsilon_n)_{n\geq 0}$ converging to $0$, we deduce by Lemma \ref{lem:phiW} that $0$ is the unique accumulation point of $r_1(\varepsilon_n)$. Then $\lim_{\varepsilon\to 0}r_1(\varepsilon)=0$.
    
    Since $V\in(0,1)$ a.s., we deduce from \eqref{eq:fixed-point-W-dimd} and $K^{2d}=I_{2d-1}$ that for $\NRM{t}=r$, $k\in\BRA{0,\ldots, 2d-1}$,
      \begin{equation*}
\ABS{\varphi_{K^k\bW}(t)}
\leq \dE\SBRA{\psi_\bW\PAR{rV^{a}}\psi_\bW\PAR{r(1-V)^{a}}}.
    \end{equation*}
    Then
    \begin{equation}\label{eq:majo-phi-V-dimd}
    \psi_{\bW}(r)\leq \dE\SBRA{\psi_\bW\PAR{rV^{a}}\psi_\bW\PAR{r(1-V)^{a}}}.
    \end{equation}
    Introducing $\phi_n(r)=\psi_\bW(rV_1^a\cdots V_n^a)$  with $(V_1,\ldots,V_n)$ i.i.d.\ copies of $V$, we obtain by induction
   \[
   \psi_\bW(r)\leq \dE[\psi_\bW(rV^a)]\leq \phi_n(r).
   \]
 The remainder of the proof is exactly as in Step~$3$ of the proof of Theorem~\ref{thm:fct-carac-X}. We then conclude by contradiction that the limit \eqref{eq:reasoning_pa} should be $0$.
\end{proof}

\begin{lem} \label{lem:majo-phiW-dim3}
    Let $\bW$ be a solution of the fixed-point equation \eqref{eq:fixed-point-W-dimd} satisfying the assumptions of Theorem~\ref{thm:density-W-dimd}. Then for all $ k\geq 1$, there exists a constant $C>0$ such that for all non-zero $ t\in\dR^{2d-1}$, \eqref{eq:upper_bound_cf} holds.
\end{lem}

\begin{proof}
     Let $\varepsilon\in(0,1)$ and $T_\varepsilon>0$ such that $\max_{k\in \BRA{0,\ldots,2d-1}}\ABS{\varphi_\bW\PAR{\PAR{K^k}^Tt}}\leq\varepsilon$ as soon as $\NRM{t}\geq T_\varepsilon$. The existence (finiteness)\ of $T_\varepsilon$ is guaranteed by Lemma~\ref{lem:phiW-to-0}. We have from Equation \eqref{eq:fixed-point-phiW} that for any non-zero $t$,
\begin{align*}
     \ABS{\varphi_\bW(t)} 
     &\leq \varepsilon\dE\SBRA{\ABS{\varphi_\bW(tV^a )}}+\dP((1-V)^a \NRM{t}\leq T_\varepsilon)\\
     &\leq \varepsilon\dE\SBRA{\ABS{\varphi_\bW(tV^a )}}+ \frac{1}{1-\rho a}\PAR{\frac{T_\varepsilon}{\NRM{t}}}^{\rho},
 \end{align*}
 for any $\rho\in\PAR{0,1/a}$ and $\NRM{t}>T_\epsilon$.
The conclusion follows using the same arguments as in the proof of Theorem~\ref{thm:fct-carac-X} (Step~4).
\end{proof}

\begin{proof}[Proof of Theorem~\ref{thm:density-W-dimd}]
The first point is proved in Lemma~\ref{lem:majo-phiW-dim3}, and choosing $\rho\in(0,1/a)$ and $k\geq 1$ such that $k\rho>2d-1$, we deduce that $\varphi_\bW$ is integrable on $\dR^{2d-1}$. By Fourier inversion theorem, we deduce that the random vector $\bW$ has a bounded density with respect to the Lebesgue measure on $\dR^{2d-1}$. Then, since $k\geq 1$ may be chosen arbitrarily large, we obtain that the density is smooth on $\dR^{2d-1}$.   
\end{proof}

\subsection{Support of the urn limit}
\label{sec:support}

The aim of this section is to prove that the support of the process $\bW$, defined in Theorem~\ref{thm:Janson-dimd}, is $\dR^{2d-1}$. The study of the support is not difficult for $d=2$, but as we will see it is more challenging for a general dimension $d$.
Consequently, we first show the result in dimension $2$. Then  we explain how we can study this support for any dimension $d\geq 1$, but give a complete proof only in the case $d=3$.

\paragraph{Support of the urn limit when $d=2$}

We are interested in studying non-null solutions of the fixed-point equation \eqref{eq:fixed-point-W-dimd}, with $d=2$ and
\[
     K=\begin{pmatrix}
         -1&-1&-1\\
         1&0&0\\
         0&1&0
     \end{pmatrix}\quad \text{(see \eqref{eq:def_matrix_K})}
\]
for a general probability vector $(p_k)_{0\leq k\leq 3}$, and then in our specific case (see Proposition~\ref{prop:fixed-point-dimd})
\begin{equation}
    \label{eq:specific_case}
    (p_0,p_1,p_2,p_3)=(p,\tfrac{1-p}{3},\tfrac{1-p}{3},\tfrac{1-p}{3}),\quad  \text{with\ }p=\tfrac{3a+1}{4}.
\end{equation}

\begin{prop}\label{prop:supportW-dim2} Let $a\in(1/2,1)$ and $\bW$ be a non-null solution of the fixed-point equation \eqref{eq:fixed-point-W-dimd} with a general  probability vector $(p_k)_{0\leq k\leq 3}$ with positive coordinates. 
\begin{enumerate}
    \item The support of $\bW$ is either $\mathrm{Span}\{(1,-1,1)\}$, or $\mathrm{Span}\{(1,0,-1),(0,1,0)\}$, or $\dR^3$.
\item In our specific case \eqref{eq:specific_case}, when $\bW$ is the limit \eqref{eq:conv-Udimd} of the urn process starting from $\bU(0)=\be_1$, then 
$\Supp(\bW)=\dR^3$ and $\bW$ has a smooth bounded density on $\dR^3$.
\end{enumerate}
\end{prop}

\begin{proof}
Let $\bw\in\mathrm{Supp }(\bW)$ with $\bw\neq 0$. 
Then by Equation \eqref{eq:fixed-point-W-dimd}, since $\dP(B=I_3)>0$, $t^{a}\bw+(1-t)^{a}\bw\in\mathrm{Supp } (\bW)$ for any $t\in[0,1]$. We deduce that $\mathrm{Supp }(\bW)$ is not a singleton ($a\neq 1$).
Since $\dP(B=K^k)>0$ for $k\in\BRA{1,2,3}$, we also have $K\bw, K^2\bw, K^3\bw\in\mathrm{Supp } (\bW)$. 
We remark that if $\bw=(x,y,z)$, then the determinant
\begin{equation}
\label{eq:formula_determinant}
   \mathrm{det}\PAR{\bw,K\bw,K^2\bw}=(x+z)\PAR{(x+y)^2+(y+z)^2}. 
\end{equation}

We now examine three cases. If first $\bw\in \Span\{(1,-1,1)\}$, then $K\bw=-\bw$. Using similar arguments as in dimension $1$, we deduce that 
\begin{equation*}
    \Span\{\bw\}=\Span\{(1,-1,1)\}\subset \Supp (\bW).
\end{equation*}
We also note that 
$K(\mathrm{Span}\{(1,-1,1)\})=\mathrm{Span}\{(1,-1,1)\}$, which implies that $\Supp(\bW)=\mathrm{Span}\{(1,-1,1)\}$ is possible.

If now $\bw\in \mathrm{Span}\{(1,0,-1),(0,1,0)\}$, then $K^2\bw=-\bw$. We deduce 
\begin{equation*}
   \mathrm{Span}\{\bw,K\bw\}=\mathrm{Span}\{(1,0,-1),(0,1,0)\}\subset \mathrm{Supp }(\bW).
\end{equation*}
We note that $K(\Span\{(1,0,-1),(0,1,0)\})=\Span\{(1,0,-1),(0,1,0)\}$, and then $\Supp(\bW)=\Span\{(1,0,-1),(0,1,0)\}$ is also possible.

If finally $\bw\notin \mathrm{Span}\{(1,-1,1)\}$ and $\bw\notin \mathrm{Span}\{(1,0,-1),(0,1,0)\}$, then $\bw,K\bw,K^2\bw$ are  linearly independent by \eqref{eq:formula_determinant}. By \eqref{eq:fixed-point-W-dimd}, for every $\bv,\bv'\in \BRA{\bw,K\bw,K^2\bw}$ and all $t\in[0,1]$, $t^{a}\bv+(1-t)^{a}\bv'\in \mathrm{Supp }(\bW) $. 
Then, as in dimension $1$, we deduce $[\bv,\bv']\subset \mathrm{Supp }(\bW)$ and $2^{1-a}\bv\in \mathrm{Supp }(\bW)$ (taking $t=1/2$). 
We obtain that 
\[
\mathrm{Span}_+\{\bw,K\bw,K^2\bw\}:=\BRA{\sum_{k=0}^2\lambda_kK^{k}\bw: \lambda_k\in\dR_+}\subset \mathrm{Supp }(\bW).
\]
Moreover, we observe that 
$
\bw+K\bw$ and $ K\bw+K^2\bw$ belong to $ \mathrm{Span}\{(1,0,-1),(0,1,0)\}
$, then $K^2\bw+K^3\bw=-(\bw+K\bw)$ and $K^3\bw+\bw= -(K\bw+K^2\bw)$. We  deduce that $-\bw,-K\bw, -K^2\bw\in \mathrm{Supp }(\bW)$, and consequently
\begin{equation*}
   \mathrm{Span }\{\bw,K\bw,K^2\bw\}=\dR^3=\mathrm{Supp }(\bW).
\end{equation*}
The first result of the proposition is proved.

Let us now consider the specific case of the asymptotic urn process starting from $\bU(0)=\be_1$. By Theorem \ref{thm:Janson-dimd}, we have $\dE[\bW]=\frac{-1}{4\Gamma(a+1)}(1,1,1)^T$, which means that all coordinates of $\bW$ have the same expectation. We easily deduce that, under this condition, the only possible support for $\bW$ is $\Supp(\bW)=\dR^3$. We conclude by Theorem~\ref{thm:density-W-dimd}.
\end{proof}
\paragraph{Support of $\bW$ for a general dimension $d\geq 1$}

As in the case $d=2$, the support of a general non-null solution $\bW$ of the  fixed-point equation \eqref{eq:fixed-point-W-dimd} is not unique (we exclude the case $\bW=0$ a.s.).

Let $\bw\in\mathrm{Supp}(\bW)$ with $\bw\neq 0$. If $\bw$ is such that its determinant
\begin{equation}
    \label{eq:det_to_study}
    \det(\bw,K\bw,\ldots,K^{2d-2}\bw)
\end{equation}
is non-zero, then using the same proof as in dimension $2$ (see Proposition \ref{prop:supportW-dim2}), we deduce that $\mathrm{Supp}(\bW)=\dR^{2d-1}$. 
We thus want to identify all vectors $\bw\in\dR^{2d-1}$ which have the property that the determinant \eqref{eq:det_to_study} is (non-)zero.
To that purpose, it is convenient to introduce the Krylov space $\Span\{\bw,K\bw,\ldots,K^{2d-2}\bw\}$ generated by $\bw$ (see e.g.\ \cite{Cheon-Hana2013}). Let us mention that for all $\ell\in\BRA{0,\ldots, 2d-1}$, \begin{equation*}
\Span\{(K^j\bw)_{0\leq j\leq 2d-1, j\neq \ell}\}=\Span\{\bw,K\bw,\ldots,K^{2d-2}\bw\},
\end{equation*}
since $K^{2d}=I_{2d-1}$.

\begin{prop}
\label{prop:det-KW}
    For $\bw=(w_1,\ldots,w_{2d-1})\in\dR^{2d-1}$, the determinant \eqref{eq:det_to_study} is zero if and only if at least
one of the following linear equations is satisfied
\begin{equation*}
    \sum_{k=1}^{2d-1}w_k(\ee^{i\frac{kj\pi}{d}}-1)=0,
\end{equation*}
with $j\in\BRA{1,\ldots,2d-1}$.
The codimension of $\mathrm{Span}\{\bw,K\bw,\ldots,K^{2d-2}\bw\}$ is exactly the number of the above equations satisfied.
\end{prop}

\begin{proof}
Let $\PAR{\bbf_i}_{1\leq i\leq 2d-1}$ be the canonical basis of $\dR^{2d-1}$. We first prove that for $\bw =\bbf_1$, the determinant \eqref{eq:det_to_study} is non-zero. To that aim, we study the matrix $K$ in \eqref{eq:def_matrix_K}.  We have
$\bw=\bbf_1$ and for $\ell\in\BRA{1,\ldots,2d-2}$,  $K^\ell \bbf_1=-\bbf_\ell+\bbf_{\ell+1}$. As a consequence 
\begin{equation*}
    \mathrm{Span}\{\bbf_1,K\bbf_1,\ldots,K^{2d-2}\bbf_1\}=\dR^{2d-1},
\end{equation*}
which proves our claim about \eqref{eq:det_to_study} being non-zero for $\bw =\bbf_1$.

As a second step, we give a formula for the determinant \eqref{eq:det_to_study}. The matrix $K$ satisfies $K^{2d}=I_{2d-1}$. It is easy to see that $1$ is not an eigenvalue of $K$. The eigenvalues in $\dC$ of $K$ are $\eta_k=\eta^{k}$ with $\eta=\ee^{i\pi/d}$. Solving $K \bz=\eta^k \bz$, we obtain the respective eigenvectors, which are given by
\begin{equation*}
   \bz_k=\bigl(\eta^{-k},\eta^{-2k},\ldots,\eta^{-(2d-1)k}\bigr)^T\in\dC^{2d-1},
\end{equation*}
for $k\in\BRA{1,\ldots, 2d-1}$. When taking $k=d$, we observe that $-1$ is an eigenvalue with associated eigenvector $\bz_d=\PAR{-1,1,-1,\ldots,1,-1}$. It is the unique real eigenvalue of $K$. We introduce the matrix $P=(\bz_1,\ldots,\bz_{2d-1})$, and for $\bw\in\dR^{2d-1}$, we define
\begin{equation}
\label{eq:def-u}
   \bu:=P^{-1}\bw=(u_1,\ldots,u_{2d-1}).
\end{equation}
As $K=P\mathrm{diag}(\eta,\eta^2,\ldots,\eta^{2d-1})P^{-1}$,
we deduce that
\begin{equation*}
    \det(\bw,K\bw,\ldots,K^{2d-2}\bw)=\det(P)\det\PAR{\PAR{\eta^{(i-1)(k-1)}}_{1\leq i,k\leq 2d-1}}\prod_{j=1}^{2d-1}u_j.
\end{equation*}
Our previous computation with $\bw=\bbf_1$ shows that the second determinant in the right-hand side of the above identity is non-zero (this could also be seen using a direct computation). Accordingly, the determinant \eqref{eq:det_to_study} is zero if and only if there exists $j$ such that $u_j=0$. Moreover, we have
\begin{equation*}
    \dim\bigl(\mathrm{Span}\{\bw,K\bw,\ldots,K^{2d-2}\bw\}\bigr)=\mathrm{card}\BRA{j:u_j\neq 0}.
\end{equation*} 

We now express the matrix $P^{-1}$, as it gives an explicit relation between $\bu$ and $\bw$, see \eqref{eq:def-u}. More precisely, we show that
\begin{equation}
\label{eq:inverse-P}
P^{-1}=\frac{1}{2d}\begin{pmatrix}
    \overline{\bz}_1^T-1\\
    \vdots\\
    \overline{\bz}_{2d-1}^T-1
\end{pmatrix}.
\end{equation}
By definition of $P$, we have 
\begin{equation}\label{eq:w=Pu}
\bw=P\bu=\sum_{j=1}^{2d-1}u_j\bz_j.
\end{equation}
Moreover, for all $k,\ell\in\BRA{1,\ldots,2d-1}$,
\begin{equation*}
   \SCA{\bz_k,\bz_\ell}=\sum_{j=1}^{2d-1}\overline{\bz}_{j,k} \bz_{j,\ell}
=\sum_{j=1}^{2d-1}\eta^{j(k-\ell)}
=\eta^{k-\ell}\frac{1-\eta^{(2d-1)(k-\ell)}}{1-\eta^{k-\ell}}=\begin{cases}
   -1&\text{ if }k\neq \ell,\\
   2d-1&\text{ is }k=\ell,
\end{cases}
\end{equation*}
because $\eta^{2d(k-\ell)}=1$. We deduce from \eqref{eq:w=Pu} that
\begin{equation}\label{eq:SCA(v,w)}
\SCA{\bz_k,\bw}=\sum_{j=1}^{2d-1}u_j\SCA{\bz_k,\bz_j}=-\sum_{j=1}^{2d-1}{u_j}+2du_k
\end{equation}
and
\[
\SCA{\sum_{k=1}^{2d-1}\bz_k,\bw}=\sum_{k=1}^{2d-1}u_k,
\]
with $\sum_{k=1}^{2d-1}\bz_k=(-,1,\ldots,-1)^T$. Then Equation \eqref{eq:SCA(v,w)} implies $2du_k=\SCA{\bz_k,\bw}-\sum_{j=1}^{2d-1}w_j=(\overline{\bz}_k-1)^T\bw$ and we obtain \eqref{eq:inverse-P}.

To conclude, we recall that  for $j<d$, $\eta^j=\overline{\eta}^{2d-j}$, thus $\bz_j=\overline{\bz}_{2d-j}$. Moreover, since $\bz_d=\sum_{k=1}^{2d-1}(-1)^ke_k$, we have $\overline{\bz}_d-1=-2\sum_{j=1}^{d}\bbf_{2j-1}$. By definition \eqref{eq:def-u} of $\bu$, we derive
\begin{itemize}
    \item for $j\in\BRA{1,\ldots, d-1}$, $u_j=0$ $\Longleftrightarrow$ $u_{2d-j}=0$ $\Longleftrightarrow$  $\ds{\sum_{k=1}^{2d-1}w_k(\ee^{i\frac{kj\pi}{d}}-1)}=0$; 
    \item $u_d=0$ $\Longleftrightarrow$ $\sum_{j=1}^{d}w_{2j-1}=0$.
\end{itemize}
The proposition is proved.
\end{proof}

As a first consequence of Proposition~\ref{prop:det-KW}, we recover the result proved in Proposition~\ref{prop:supportW-dim2} on the support of the variable $\bW$ in dimension $2$. 

\begin{cor}
\label{cor:SuppW2}
In dimension $2$, discarding the case when $\bW$ is zero, the only possibilities for $\Supp(\bW)$ are $\Span\{\bbf_1-\bbf_2+\bbf_3\}$, $\Span\{\bbf_2,\bbf_1-\bbf_3\}$ and $\dR^3$.
\end{cor}

\begin{proof}
    For $d=2$, one has $\eta=i$ and with \eqref{eq:inverse-P} we obtain
 \begin{equation*}
 P^{-1}=\frac{1}{4}\begin{pmatrix}
     -1+i&-2&-1-i\\
     -2&0&-2\\
     -1-i&-2&-1+i
 \end{pmatrix}
 \end{equation*}
Then using Proposition~\ref{prop:det-KW} (or directly \eqref{eq:def-u}), we have 
\begin{itemize}
\item 
$u_1=0$ $\Longleftrightarrow$ $u_3=0$ $\Longleftrightarrow$ $w_1=-w_2=w_3$; 
\item $u_2=0$ $\Longleftrightarrow$ $w_1=-w_3$. \qedhere
\end{itemize} 
\end{proof}

More interestingly, Proposition~\ref{prop:det-KW} also allows us to describe the possible supports of $\bW$ in dimension three.
\begin{cor}
\label{cor:SuppW3}
In dimension $d=3$, discarding the case when $\bW$ is zero, $\Supp(\bW)$ is one of the following vector spaces:
\begin{enumerate}[label={\rm(\roman*)}]
    \item\label{it:dim1_1/1}$\Supp(\bW)=\Span\{\bbf_1-\bbf_2+\bbf_3-\bbf_4+\bbf_5\}$;
    \item\label{it:dim2_1/2}
       $\Supp(\bW)=\Span\{\bbf_2-\bbf_3+\bbf_5,-\bbf_1+\bbf_3-\bbf_4\}$; 
        \item\label{it:dim2_2/2} $\Supp(\bW)=\Span\{\bbf_2+\bbf_3-\bbf_5,-\bbf_1+\bbf_3+\bbf_4\}$;
        \item\label{it:dim3_1/2} $\Supp(\bW)=\Span\{\bbf_3-2\bbf_4+2\bbf_5,-\bbf_1+\bbf_4-2\bbf_5,2\bbf_1-\bbf_2+\bbf_5\}$; 
    \item\label{it:dim3_2/2} $\Supp(\bW)=\Span\{\bbf_3,-\bbf_1+\bbf_4,-\bbf_2+\bbf_5\}$;
     \item\label{it:dim4_1/1} $\Supp(\bW)=\Span\{\bbf_2,-\bbf_1+\bbf_3,-\bbf_2+\bbf_4,-\bbf_3+\bbf_5\}$; 
    \item\label{it:dim5_1/1} $\Supp(\bW)=\dR^{5}$.
\end{enumerate}
\end{cor}

\begin{proof}
If $d=3$ then $\eta=\ee^{i\pi/3}$ and with \eqref{eq:inverse-P} 
 \begin{align*}
 P^{-1}&=\frac{1}{6}\begin{pmatrix}
     \ee^{i\pi/3}-1&\ee^{i2\pi/3}-1&-2&\ee^{-i2\pi/3}-1&\ee^{-i\pi/3}-1\\
      \ee^{i2\pi/3}-1& \ee^{-2i\pi/3}-1&0& \ee^{i2\pi/3}-1& \ee^{-i2\pi/3}-1\\
     -2&0&-2&0&-2\\
     \ee^{-i2\pi/3}-1&\ee^{i2\pi/3}-1&0&\ee^{-i2\pi/3}-1&\ee^{i2\pi/3}-1\\
 \ee^{-i\pi/3}-1&\ee^{-i2\pi/3}-1&-2&\ee^{2\pi/3}-1&\ee^{i\pi/3}-1
 \end{pmatrix}.
 \end{align*}
 We notice that 
\begin{itemize}
\item $u_1=0$ $\Longleftrightarrow$ $u_5=0$ $\Longleftrightarrow$ 
\begin{align*}
&w_1+3w_2+4w_3+3w_4+w_5=0\quad \text{and} \quad w_1+w_2-w_4-w_5=0\\
\Longleftrightarrow\quad 
&w_1+w_2-w_4-w_5=0\quad \text{and} \quad w_2+2w_3+2w_4+w_5=0;
\end{align*}
\item $u_2=0$ $\Longleftrightarrow$ $u_4=0$ $\Longleftrightarrow$ 
\begin{align*}
&w_1+w_2+w_4+w_5=0\quad \text{and} \quad w_1-w_2+w_4-w_5=0\\
\Longleftrightarrow\quad 
&w_1+w_4=0=w_2+w_5;
\end{align*}
\item $u_3=0$ $\Longleftrightarrow$ $w_1+w_3+w_5=0$.
\end{itemize}
Accordingly, we obtain that the Krylov space  generated by $\bw$ may be of
\begin{itemize}
    \item \textit{dimension $1$}, when only $u_1=u_2=0$. Then  $\bw\in\Span\{\bbf_1-\bbf_2+\bbf_3-\bbf_4+\bbf_5\}$. We also notice that $K(\bbf_1-\bbf_2+\bbf_3-\bbf_4+\bbf_5)=-(\bbf_1-\bbf_2+\bbf_3-\bbf_4+\bbf_5)$. We obtain \ref{it:dim1_1/1}.
    \item \textit{dimension $2$}, when only $u_1=u_3=0$ or $u_2=u_3=0$. We observe that 
    \begin{equation*}
    \{\bw\in \dR^{2d-1}:u_1=0=u_3\}=\Span\{\bbf_2-\bbf_3+\bbf_5,-\bbf_1+\bbf_3-\bbf_4\}\qquad \text{(case \ref{it:dim2_1/2})}
    \end{equation*}
    and
    \begin{equation*}
        \{\bw\in \dR^{2d-1}:u_2=0=u_3\}=\Span\{\bbf_2+\bbf_3-\bbf_5,-\bbf_1+\bbf_3+\bbf_4\}\qquad \text{(case \ref{it:dim2_2/2})}.
    \end{equation*} 
       Moreover, we have 
       \[
        K(\bbf_2-\bbf_3+\bbf_5)=-\bbf_1+\bbf_3-\bbf_4,\quad K^2(\bbf_2-\bbf_3+\bbf_5)=-(\bbf_2-\bbf_3+\bbf_5)-(-\bbf_1+\bbf_3-\bbf_4),\]
        and 
         \[K(\bbf_2+\bbf_3-\bbf_5)=-\bbf_1+\bbf_3+\bbf_4,\quad K^2(\bbf_2+\bbf_3-\bbf_5)=-\bbf_1+\bbf_3+\bbf_4-(\bbf_2+\bbf_3-\bbf_5).\]
    \item \textit{dimension $3$}, when only $u_1=0$ or $u_2=0$. We have 
    \begin{align*}
        \{\bw\in \dR^{2d-1}:u_1=0\}&=\Span\{\bbf_3-2\bbf_4+2\bbf_5,-\bbf_1+\bbf_4-2\bbf_5,2\bbf_1-\bbf_2+\bbf_5\} & \text{(case \ref{it:dim3_1/2})},\\
        \{\bw\in \dR^{2d-1}:u_2=0\}&=\Span\{\bbf_3, -\bbf_1+\bbf_4,-\bbf_2+\bbf_5\}
        & \text{(case \ref{it:dim3_2/2})}.\phantom{i}
    \end{align*}
    On one hand, we have 
    \begin{align*}
        K(\bbf_3-2\bbf_4+2\bbf_5)&=-\bbf_1+\bbf_4-2\bbf_5,\\ 
K^2(\bbf_3-2\bbf_4+2\bbf_5)&=2\bbf_1-\bbf_2+\bbf_5\\
        K^3(\bbf_3-2\bbf_4+2\bbf_5)&=2\bbf_1+2\bbf_2-\bbf_3\\
        &=-(\bbf_3-2\bbf_4+2\bbf_5)-2(-\bbf_1+\bbf_4-2\bbf_5)-2(2\bbf_1-\bbf_2+\bbf_5),
    \end{align*}
     and in the other hand,
   \begin{equation*}
        K\bbf_3=-\bbf_1+\bbf_4, \quad K^2\bbf_3=-\bbf_2+\bbf_5, \quad K^3\bbf_3=-\bbf_3.
    \end{equation*}
    \item \textit{dimension $4$}, when only $u_3=0$. We have 
    \begin{equation*}\{\bw\in \dR^{2d-1}:u_3=0\}=\Span\{\bbf_2,-\bbf_1+\bbf_3,-\bbf_2+\bbf_4,-\bbf_3+\bbf_5\}
    \end{equation*}
and $K\bbf_2=-\bbf_1+\bbf_3$, $K^2\bbf_2=-\bbf_2+\bbf_4$, $K^3\bbf_2=-\bbf_3+\bbf_5$ and $K^4\bbf_2=-\bbf_4=-\bbf_2-K^2\bbf_2$. This corresponds to case~\ref{it:dim4_1/1}.
    \item \textit{dimension $5$},  when $u_j\neq0$ for all $j$, and $\Supp(K)=\dR^{5}$ (case \ref{it:dim5_1/1}).\qedhere
\end{itemize}
\end{proof}

We now focus on $d=3$.
\begin{prop}\label{prop:densityW-dim3}
In dimension $3$, by  \eqref{eq:def_a_dim-d}, we have $a=\frac{6p-1}{5}$. Assume that $\frac{1}{2}<a<1$. The process $\bW$ of the urn limit for $d=3$ starting from a ball of color $1$, defined in Theorem~\ref{thm:Janson-dimd}, 
has a smooth bounded density on $\dR^{5}$, and its support is $\dR^5$.
\end{prop}

\begin{proof}
By Theorem~\ref{thm:density-W-dimd}, we only have to prove that $\Supp(\bW)=\dR^5$.

Let  $\bY_{(1)}\in\dR^{6}$ be the asymptotics of the urn process starting from a unique ball of color $1$, introduced in Theorem~\ref{thm:Janson-dimd}. We recall that $\bY_{(1)}=\sum_{i=2}^{6}W_i\bv_i$ with $\bv_i=\frac{1}{2}(\be_1-\be_i)$ for $i\geq 2$. 
Then, by Equation~\eqref{eq:expectation-W-2d},
\begin{equation*}
    \dE[\bY_{(1)}]=\frac{1}{2}\dE\SBRA{\begin{pmatrix}
        W_2+\ldots+W_{6}\\
        -W_2\\
        \vdots\\
        -W_{6}
    \end{pmatrix}} = \frac{-1}{6\Gamma(a+1)}\begin{pmatrix}-5\\1\\\vdots\\1\end{pmatrix}.
\end{equation*}
 Let us consider all possible supports of $\bW$ described in Corollary~\ref{cor:SuppW3}. Due to the fact that $\dE[W_2]=\ldots=\dE[W_6]$, all  cases happen to be impossible, except Case~\ref{it:dim5_1/1}. Indeed,
\begin{itemize}
    \item[\ref{it:dim1_1/1}] $\Supp(\bW)=\Span\{\bbf_1-\bbf_2+\bbf_3-\bbf_4+\bbf_5\}$ implies $W_2=-W_3$, while we should have $\dE[W_2]=\dE[W_3]$;
    \item[\ref{it:dim2_1/2}] $\Supp(\bW)=\Span\{\bbf_2-\bbf_3+\bbf_5,-\bbf_1+\bbf_3-\bbf_4\}$ implies $W_4=-W_2-W_3$, while we should have $\dE[W_4]=\dE[W_2]=\dE[W_3]$;
    \item[\ref{it:dim2_2/2}] $\Supp(\bW)=\Span\{\bbf_2+\bbf_3-\bbf_5,-\bbf_1+\bbf_3+\bbf_4\}$ implies $W_2=-W_5$, but we should have $\dE[W_2]=\dE[W_5]$;
    \item[\ref{it:dim3_1/2}] $\Supp(\bW)=\Span\{\bbf_3-2\bbf_4+2\bbf_5,-\bbf_1+\bbf_4-2\bbf_5,2\bbf_1-\bbf_2+\bbf_5\}$ implies $W_5=-W_2-2W_3-2W_4$ and we should have $\dE[W_5]=\dE[W_4]=\dE[W_2]=\dE[W_3]$;
    \item[\ref{it:dim3_2/2}] $\Supp(\bW)=\Span\{\bbf_3,-\bbf_1+\bbf_4,-\bbf_2+\bbf_5\}$ implies $W_2=-W_5$ and we should have $\dE[W_2]=\dE[W_5]$;
   \item[\ref{it:dim4_1/1}] $\Supp(\bW)=\Span\{\bbf_2,-\bbf_1+\bbf_3,-\bbf_2+\bbf_4,-\bbf_3+\bbf_5\}$ implies $W_4=-W_2-W_6$ and we should have $\dE[W_2]=\dE[W_4]=\dE[W_6]$.
\end{itemize}
Consequently, $\Supp(\bW)=\dR^5$, which means that Theorem~\ref{thm:density-W-dimd} is satisfied and the result is proved.
\end{proof}

We see in the above proof of Proposition~\ref{prop:densityW-dim3}  that we do not need to describe exactly all the possibilities for $\Supp(\bW)$ to obtain the result, but we only need a well-adapted relationship between the coordinates of $\bW$ when the support is not $\dR^{2d-1}$. Even if it is effective, the method presented in this article does not seem optimal, because we are looking for too much information on the general solutions of the fixed-point equation \eqref{eq:fixed-point-W-dimd} to deduce the support of $\bW$ in the specific case of the asymptotic urn process.

\subsection{The distribution of superdiffusive limit of the MERW}
\label{subsec:application_density_MERW}

We now focus on the asymptotic limit $\bL\in\dR^d$  of the superdiffusive MERW to prove Theorem~\ref{cor:densityL-dim3}. We only explain the ideas when the first step of MERW is a.s.\ in direction $\be_1$. The general case will follow immediately, because $\bL$ is a discrete mixture of the variables $\bL_{(k)}$, $k\in\BRA {1,\ldots,2d}$, see \eqref{eq:mixture}.

We recall the relation \eqref{eq:L-W-dimd} between $\bL$ and the urn asymptotic $\bW$, defined in Theorem~\ref{thm:Janson-dimd}. As mentioned earlier
(see page \pageref{eq:L-W-dimd}), if $\Supp(\bW)=\dR^{2d-1}$, we immediately deduce that $\bL$ admits a smooth positive density with respect to the Lebesgue measure in $\dR^d$. As we do not have a general result on the support of the solution of the fixed-point equation \eqref{eq:fixed-point-W-dimd}, for any dimension $d\geq2$, we cannot have a general result for $\bL$. However, by Propositions~\ref{prop:supportW-dim2} and \ref{prop:densityW-dim3} in dimension $2$ and $3$, Theorem~\ref{cor:densityL-dim3} is easily deduced from Theorem~\ref{thm:density-W-dimd}, and Corollaries~\ref{cor:SuppW2} and~\ref{cor:SuppW3}.

\subsection{Moments of the asymptotic distributions in high dimension}
\label{sec:moments-dimd}

We now study the moments of the limit urn process $\bY$ and the limit $\bL$ of the MERW. To that aim, we use the fixed-point equation \eqref{eq:fixed-point-X-color-vector} and not \eqref{eq:fixed-point-W-dimd}, because in \eqref{eq:fixed-point-W-dimd} the norm of the matrix $B$ is not smaller than $1$, which do not allow us to have good estimates of the moments of $\bW$.

When the first step probability vector $q$ of the MERW is the uniform distribution on the $2d$-directions,  it is proved in \cite[Thm~3.8]{BercuLaulin2019} that $\dE[\bL]=0$ and
\begin{equation*}
   \dE\SBRA{\bL\bL^T}=\frac{1}{d(2a-1)\Gamma(2a)}I_d.
\end{equation*}

\begin{thm}
\label{thm:dim_d_moments}
For $d\geq 1$, let $\bY$ be the asymptotic urn process defined in Theorem~\ref{thm:Janson-dimd}, and $\bL$ be  the asymptotics of the MERW with memory parameter $p$ and first step distribution $q$, defined in Theorem~\ref{thm:limitERW}. Then $\bY$ and $\bL$ have finite moments of all order, and their probability distribution are determined by their moments.
Moreover, 
    the moment-generating function of $\bY$, $\bL$, $\NRM{\bY}$ and $\NRM{\bL}^2$ are respectively defined on $\dR^{2d}$, $\dR^{d}$ and $\dR$: for all $t\in\dR^{2d}$, $r\in\dR^d$ and $s\in\dR$,
    \begin{equation*}
       \dE\SBRA{\ee^{\SCA{t,\bY}}}<\infty, \quad \dE\SBRA{\ee^{s\NRM{\bY}^2}}<\infty,\quad
  \dE\SBRA{\ee^{\SCA{r,\bL}}}<\infty \quad \text{and}\quad \dE\SBRA{\ee^{s\NRM{\bL}^2}}<\infty.
    \end{equation*}
\end{thm}

To prove this theorem, we use the following extended Carleman's criterion (see \cite[Thm~2.3]{deJeu03}) on the distribution of $\bY$.
\begin{thm}[Extended Carleman's criterion]\label{thm:carleman-dimd}
Let $\bX=(X_1,\ldots, X_{D})$ be a random vector with finite moments of all order. If, for any $i\in\BRA{1,\ldots,D}$, $\sum_{\ell\geq 1}\PAR{\dE[|X_i|^{2\ell}]}^{-1/2\ell}=\infty$, the distribution of $\bX$ is uniquely determined by its moments.
\end{thm}

\begin{proof}[Proof of Theorem~\ref{thm:dim_d_moments}]
We study the process $\bY$; the result for $\bL$ will easily follow from \eqref{eq:L-W-dimd}.
We first consider the random vector $\bY^j = \PAR{Y_{j,(k)}}_{1\leq k\leq 2d}$ introduced in Section~\ref{sec:fixed-point-dimd}, where  $Y_{j,(k)}$ is the $j$-th coordinate of the limiting vector $\bY_{(k)}$ starting with one ball of the $k$-th color. As noticed in that section, this vector is solution to the fixed-point equation~\eqref{eq:fixed-point-X-color-vector}, recalled below
\begin{equation*}
    \bY^{j} = V^a \bY^{j,(1)}+(1-V)^a A \bY^{j,(2)},
\end{equation*}
where $A=(A_{i_j})$ has the same law as the random replacement matrix \eqref{eq:replacement-matrix-A} of the urn process.

It is known \cite[Thm~3.9 (i)]{Janson2004} that the asymptotic urn process has finite second-order moment. 
 We also easily note that for all $\by\in \dR^{2d}$, $\NRM{A\by}=\NRM{\by}$. Then, for 
 \begin{equation*}
     C_j=\Gamma(a+1)\dE\SBRA{\NRM{\bY^j}^2}^{1/2},
 \end{equation*}
 using exactly the same proof as for Lemma~\ref{lem:GammaMomentsbound}, we prove by induction that for all $j\in\BRA{1,\ldots,2d}$ and all $\ell\geq 1$,
 \begin{equation}\label{eq:induction-gamma-dimd}
 \frac{\dE\SBRA{\NRM{\bY^j}^\ell}}{\ell!}\leq \frac{C_j^\ell}{\Gamma(\ell a+1)}\PAR{\frac{2}{a}}^{\ell-1}.
 \end{equation}
Since for all $k$, $\NRM{Y_{(k)}}\leq \sum_{j=1}^{2d}\NRM{\bY^j}$, taking $C=\max_{1\leq j\leq 2d}C_j$, the moments of $\bY$ satisfy, for all $\ell\geq 1$,
 \[
  \frac{\dE\SBRA{\NRM{\bY}^\ell}}{\ell!}\leq \frac{2d C^\ell}{\Gamma(\ell a+1)}\PAR{\frac{2}{a}}^{\ell-1}.
 \]
As in Theorem~\ref{thm:exp-moments}, we deduce  that  $\bY$ and $\NRM{\bY}$ satisfy Carleman's criterions (Theorems~\ref{thm:carleman} and \ref{thm:carleman-dimd}). 
We also easily obtain the existence of exponential moments for $\bY$ and $\NRM{\bY}^2$, using the same proofs as for Corollaries~\ref{cor:laplaceL} and \ref{cor:laplaceL^2}.
\end{proof}

\subsection*{Acknowledgements}
This project has benefited from discussions with several colleagues, who we would like to thank for their enthusiasm and their generosity. We especially thank Gerold Alsmeyer, Bernard Bercu, Jean Bertoin, Brigitte Chauvin, Tony Guttmann, Slim Kammoun, Daniela Portillo Del Valle and Nicolas Pouyanne.

\small
\bibliographystyle{abbrv}
\bibliography{biblio}
\normalsize

\end{document}